\documentclass[12pt]{amsart}
\usepackage[utf8]{inputenc}

\usepackage{BAstyle}
\usepackage[margin = 1in]{geometry}
\allowdisplaybreaks

\makeatletter
\def\@tocline#1#2#3#4#5#6#7{\relax
  \ifnum #1>\c@tocdepth 
  \else
    \par \addpenalty\@secpenalty\addvspace{#2}%
    \begingroup \hyphenpenalty\@M
    \@ifempty{#4}{%
      \@tempdima\csname r@tocindent\number#1\endcsname\relax
    }{%
      \@tempdima#4\relax
    }%
    \parindent\z@ \leftskip#3\relax \advance\leftskip\@tempdima\relax
    \rightskip\@pnumwidth plus4em \parfillskip-\@pnumwidth
    #5\leavevmode\hskip-\@tempdima
      \ifcase #1
       \or\or \hskip 1em \or \hskip 2em \else \hskip 3em \fi%
      #6\nobreak\relax
    \hfill\hbox to\@pnumwidth{\@tocpagenum{#7}}\par
    \nobreak
    \endgroup
  \fi}
\makeatother

\usepackage[breakable,many]{tcolorbox}

\tcbset{
    sharp corners,
    colback = white,
    before skip = 0.2cm,    
    after skip = 0.5cm      
}                     

\newtcolorbox{exbox}{
    sharpish corners, 
    boxrule = 0pt,
    toprule = 1pt,
    leftrule = 4.5pt, 
    enhanced,
    breakable,
    before upper={\parindent15pt},
    fuzzy shadow = {0pt}{-2pt}{-0.5pt}{0.5pt}{black!35} 
}

\newtheorem{nonexample}[theorem]{Non-Example}
\newtheorem{remark}[theorem]{Remark}

\newenvironment{BOXexample}{
\begin{exbox}\mbox{}
\begin{example}
}
{
\end{example}
\end{exbox}
}

\theoremstyle{plain}

\title{Explicit Analytic Continuation of Euler Products}
\author{Brandon Alberts}
\address{Eastern Michigan University}
\email{balbert1@emich.edu}

\begin{document}

\begin{abstract}
The generating series of a number of different objects studied in arithmetic statistics can be built out of Euler products. Euler products often have very nice analytic properties, and by constructing a meromorphic continuation one can use complex analytic techniques, including Tauberian theorems to prove asymptotic counting theorems for these objects. One standard technique for producing a meromorphic continuation is to factor out copies of the Riemann zeta function, for which a meromorphic continuation is already known.

This paper is an exposition of the ``Factorization Method" for meromorphic continuation. We provide the following three resources with an eye towards research in arithmetic statistics: (1) an introduction to this technique targeted at new researchers, (2) exposition of existing works, with self-contained proofs, that give a continuation of Euler products with constant or Frobenain coefficients to the right halfplane ${\rm Re}(s)>0$ (away from an isolated set of singularities), and (3) explicit statements on the locations and orders of all singularities for these Euler products.
\end{abstract}

\maketitle
\vspace{-1.2cm}
\setcounter{tocdepth}{1}
\tableofcontents
\newpage
\section{Introduction}

The generating Dirichlet series of arithmetic objects are very often built from Euler products. An important step in understanding the distribution of such arithmetic objects is constructing meromorphic continuations of the relevant Euler products.

There are two well-worn approaches to this problem:
\begin{enumerate}
\item \textbf{Functional Equation Method:} The primary method for constructing meromorphic continuations of a Dirichlet series is by proving the existence of a functional equation. This is done for the Riemann zeta function and many others including Dedekind zeta functions, Dirichlet and Hecke $L$-functions, and the $L$-functions associated to elliptic curves and other abelian varieties.

This type of construction betrays deep and powerful arithmetic structure, sometimes going beyond our current understanding of the subject. For this reason, it is not surprising to find that constructing a functional equation can be very involved and sometimes not possible for certain Dirichlet series.

\item \textbf{Factorization Method:} This is a more accessible approach for constructing meromorphic continuations of a Dirichlet series that builds on our knowledge of already existing examples. The Dirichlet series is decomposed in terms of (usually as a sum or product of) Dirichlet series for which one factor has a meromorphic continuation proven by other means, and the other factor converges absolutely on a larger region. In the case of Euler products, this is usually done by writing the Dirichlet series as a product of $\zeta(s)^k$ (or similar $L$-functions) with another absolutely convergent Euler product.

This method would never work in a vacuum - it requires us to build off of a collection of Dirichlet series with Euler products whose meromorophic continuations we already know.
\end{enumerate}

Often, the Factorization Method is easier than proving the existence of a functional equation, and tends to work for a larger class of Dirichlet series. Consider the following example:

\begin{BOXexample}\label{ex:intro}
Consider
\begin{align*}
    \prod_p (1 + 2p^{-s}) = \zeta(s)^2 \prod_p (1 - 3p^{-2s} + 2p^{-3s}).
\end{align*}
This can be checked locally by verifying that
\[
1+2p^{-s} = (1 - p^{-s})^{-2}(1 - 3p^{-2s}+2p^{-3s})
\]
for every prime $p$. This is a standard example of constructing a meromorphic continuation. While the left-hand side is absolutely convergent for ${\rm Re}(s) > 1$, the right-hand side is a product of $\zeta(s)^2$ (which is meromophic on all of $\C$) and an Euler product with converges absolutely on ${\rm Re}(s) > 1/2$. Not only does this extend the region of meromorphicity, we can immediately identify the only pole in this region as being $s=1$ of order $2$. Moreover, the nonzero limit
\[
\lim_{s\to 1} (s-1)^2\prod_p (1 + 2p^{-s})
\]
is given explicitly by the convergent Euler product
\[
\prod_p (1 - 3p^{-2} + 2p^{-3}).
\]
\end{BOXexample}

Example \ref{ex:intro} showcases a common target for work in arithmetic statistics: $s=1$ is the \textbf{rightmost singularity} of the Dirichlet series, which means that it is the singularity with the largest real part. In arithmetic statistics, meromorphic continuations of $\sum a_n n^{-s}$ are primarily constructed on an open neighborhood of the rightmost singularity so that the Selberg--Delange method can be used to convert the rightmost singularity into asymptotic information  for the sum of coefficients function $\sum_{n < X} a_n$. (The Selberg--Delange method is also colloquially referred to as a ``Tauberian theorem"). The location, order, and coefficient of the rightmost singularity together provide the necessary information to compute the main term in the asymptotic growth rate.

This paper has been separated into three parts in order to address three separate goals:
\begin{enumerate}[1)]
    \item We present an expository introduction to common techniques and heuristics for using the Factorization Method to compute the rightmost singularity, geared towards research in arithmetic statistics.
    
    This is done in Part \ref{part:calculations}. We intend for these sections to serve as an introduction of these methods to new researchers in the area, with a particular focus on computational examples. This section could be included with a first course in analytic number theory: a typical introduction to generating Dirichlet series provides ample background.
    
    \item We give an exposition of previous work on continuations for certain families of Euler products via the Factorization Method, together with self-contained proofs, in Part \ref{part:existence}. The results we explore have a particular focus on cases that are known to be common in arithmetic statistics research.

    We mostly focus on an older program of study, tracing back to work of Estermann \cite{estermann_1928}, for determining the largest region on which a continuation exists. This is referred to as finding the ``natural boundary" for the function. These works typically do not always provide explicit information about the locations and orders of any singularities, although such information can usually be teased out of their proofs. We will primarily focus on constructing the analytic continuations considered in \cite{estermann_1928} and in \cite{dahlquist_1952, kurokawa_1978, kurokawa_1986I, kurokawa_1986II, landau_walfisz_1920, moroz_1988} following Estermann's work. We will briefly discuss the theory of the natural boundary of Euler products in Section \ref{sec:natural_boundary}, although our focus will be on the construction of analytic continuations.
    
    The program following Estermann's work produces the data underlying analytic continuation all at once, packaged as an infinite product of powers of the Riemann zeta function (see Section \ref{sec:iterating_factorization_method} as well as the sequence $b_n(Q)$ in Section \ref{sec:history}). This is in contrast to more modern results in arithmetic statistics, which focus on explicit computation of one singularity at a time (from right to left) as is necessary to apply the Selberg--Delange method. Researchers in arithmetic statistics are typically less concerned with the natural boundary. A non-exhaustive list of such papers includes the analytic sections of \cite{alberts2021,frei-loughran-newton2019,kaplan-marcinek-takloo-bighash2015,wood2009}, where the necessary meromorphic continuations are proven as analytic lemmas. Most of these results reprove special cases of previously known meromorphic continuations, only as far to the left as necessary to reveal the rightmost singularity.

    \item Finally, in Part \ref{part:explicit} we give statements for continuations for certain families of Euler products via the Factorization Method which include explicit descriptions of the locations and orders of any singularities. In order to make these results more accessible for computation, we prove new explicit formulas for the coefficients of certain log power series in Theorem \ref{thm:log} and Theorem \ref{thm:log_trchar}.
    
    We tailor these results for use in arithmetic statistics research, specifically providing enough information for researchers to apply the Selberg--Delange method to produce an asymptotic main term, power saving error terms when possible, and secondary asymptotic terms (which correspond to the singularities to the left of the rightmost singularity) when these exist.
\end{enumerate}

Our intention is for this paper to be a good reference for research in arithmetic statistics. The Euler products considered in this paper are not an exhaustive list for which the Factorization Method will work. We restrict ourselves to the most relevant cases to research in arithmetic statistics, but with some modification these techniques will apply to numerous other cases as well. Towards this end, we include some remarks and technical lemmas in greater generality than we will use.

Readers looking for a quick reference for a general result can refer to the outline below:
\begin{itemize}
    \item \textbf{Euler products with constant coefficients.} These are Euler products of the form $\prod Q(p^{-s})$ for some complex function $Q(z)$. Some examples can be found in Section \ref{sec:examples_with_const_coef}, while a brief history of analytic continuations of these Euler products can be found in Subsection \ref{subsec:histroy_const_coef}.

    Meromorphic continuations for these Euler products to ${\rm Re}(s) > 0$ are proven in Theorem \ref{thm:existence_const_coef}. Explicit computation of the poles can be done with the assistance of Theorem \ref{thm:explicit_singularities_const_coef}.

    \item \textbf{Frobenian Euler products.} These are Euler products $\prod Q_p(p^{-s})$ for $Q_p(z)$ a collection of complex functions for which $p\mapsto Q_p(z)$ is determined by the splitting type of $p$ in some finite extension. Some examples can be found in Section \ref{sec:examples_with_Frob_coef}, while a brief history of analytic continuations of these Euler products can be found in Subsection \ref{subsec:history_Frob_coef}.

    Meromorphic continuations for these Euler products to ${\rm Re}(s) > 0$ are proven in Theorem \ref{thm:existence_Frob_coef}. Explicit computation of the poles can be done with the assistance of Theorem \ref{thm:explicit_singularities_Frob_coef}.

    \item \textbf{Factoring $\log(1-\sum x_i)$ in terms of logs of monomials.} We prove a result in linear algebra and combinatorics in Theorem \ref{thm:log}, which factors $\log(1-\sum x_i)$ into a nonegative integer linear combination of $\log(1-m)$ as $m$ ranges over the monomials in $\{x_i\}$.

    This result was essential in determining when the singularities of an Euler product are all poles. We expect that it will be useful for the study of many families of Euler products beyond those considered in this paper.

    \item \textbf{Factoring $1+{\rm tr}\rho(g)z$ in terms of characteristic polynomials.} We prove Theorem \ref{thm:log_trchar} that expresses $1+{\rm tr}\rho(g)z$ as an infinite product of characteristic polynomials for explicitly constructed representations.

    This result was essential in determining when the singularities of a Frobenian Euler product are all poles. We expect that it will be useful for the study of many of families of Euler products beyond those considered in this paper.
\end{itemize}

\section*{Acknowledgements}
The author was partially supported by an AMS-Simons Travel Grant.

The author would like to thank Robert Lemke Oliver for numerous conversations on the analytic details of this paper. The author would also like to thank Alina Bucur, Helen Grundman, Kiran Kedlaya, Daniel Loughran, Alexander Slamen, and Melanie Matchett Wood for helpful conversations and feedback.

\part{Calculating the Rightmost Singularity}\label{part:calculations}

\section{A Recipe for the Factorization Method}

In Example \ref{ex:intro}, we provided a factorization of the Euler product and verified it directly. However, in practice you will typically need to construct the decomposition from scratch. The Factorization Method for actually computing the rightmost singularity via this type of decomposition will always broadly follow the same steps:
\begin{enumerate}[(1)]
\item Identify the \textbf{minimum degree term(s)} in each Euler factor (excluding the constant term) as described in Heuristic \ref{heur}. Use these to identify $a$ and $b$ for which the heuristic predicts a pole at $s=1/a$ of order $b$.

\item Multiply the top and bottom by a \textbf{strategic Euler factor} corresponding to each minimum degree term. This factor is chosen to make the following two steps work as intended.

\item Separate the denominators, and pull them the outside. This product should match the Euler product of a Dirichlet series with a well-known meromorphic continuation, such as the Riemann Zeta function or an $L$-function.

\item Multiply out the numerators. The minimum degree terms should cancel out, and the resulting Euler product should converge absolutely on a larger region.
\end{enumerate}

Let's walk through this process to recreate the factorization in Example \ref{ex:intro}.

\begin{BOXexample}\label{ex:intro_expanded}
    We again consider the Dirichlet series
    \[
    \mathcal{Q}(s) = \prod_p (1 + 2p^{-s}).
    \]
    We now apply the four steps to the Factorization method:
    \begin{enumerate}[(1)]
    \item The minimum degree term of each Euler factor is $2p^{-s}$. The degree is $a=1$, and the leading coefficient is a constant value $b=2$.

    \item We choose the strategic Euler factor $(1-p^{-as})^b$ based on the $a$ and $b$ values from step (1). In our case this is $(1 - p^{-s})^2$, so that multiplying top and bottom gives
    \[
    \mathcal{Q}(s) = \prod_p \frac{(1 + 2p^{-s})(1-p^{-s})^2}{(1-p^{-s})^2}.
    \]
    We will see why this choice is the right one in steps (3) and (4).

    \item Pulling out the denominators, the product
    \[
    \prod_p \frac{1}{(1-p^{-s})^2}
    \]
    is recognized to be the Euler product for $\zeta(s)^2$. This is part of the reason that $(1-p^{-s})^2$ is a good choice. Thus, we have shown
    \[
    \mathcal{Q}(s) = \zeta(s)^2\prod_p (1 + 2p^{-s})(1-p^{-s})^2.
    \]

    \item Next, we multiply out the remaining numerator
    \begin{align*}
        \mathcal{Q}(s) &= \zeta(s)^2\prod_p (1 + 2p^{-s})(1-2p^{-s}+p^{-2s})\\
        &= \zeta(s)^2\prod_p (1 -3p^{-2s} + 2p^{-3s}).
    \end{align*}
    Notice that the minimum degree terms cancel out, leaving $p^{-2s}$ as the smallest remaining degree. This is the other reason that $(1-p^{-s})^2$ is a good choice. We will see in Example \ref{ex:intro_abs_conv} how to conclude that the remaining Euler product converges absolutely on the region ${\rm Re}(s)>1/2$.
    \end{enumerate}
\end{BOXexample}

We've bolded the names ``\textbf{minimum degree term(s)}" and ``\textbf{strategic Euler factor}" in the Factorization Method on purpose. This is merely an outline of a general method, not a theorem, where our goal in defining these terms is to make it so that steps (3) and (4) succeed. This can be more of an art than a science, and it depends very strongly on which ``Dirichlet series with a well-known meromorphic continuation" you plan or expect to factor out in step (3).

We will explore several families of Dirichlet series where the best choice of minimum degree terms and strategic Euler factors are well understood, and we will explain where the choices come from. The example
\[
b p^{-as} \longmapsto (1-p^{-as})^b
\]
is a good one to keep in mind as we explore more involved cases. However, we warn the reader that applying the Factorization Method to other Euler products may require some creativity.

\section{A Heuristic for the Rightmost Singularity}\label{sec:right}

Many researchers are adept at predicting the location of the rightmost pole of an Euler product, often times without needing to perform any computations. This is because they know the Factorization Method suggests that the rightmost singularity comes from the minimum degree terms, which is easily read off the Euler product in nice cases.

This has developed into a heuristic:

\begin{heuristic}\label{heur}
Consider the Euler product
\[
\mathcal{Q}(s) = \prod_p \left(1 + b_p p^{-as} + \cdots\right).
\]
If $b_p$ and the higher degree terms are ``reasonable", then each of the following hold:
\begin{enumerate}
    \item[(i)] $\mathcal{Q}(s)$ converges absolutely on the region ${\rm Re}(s)>1/a$,
    \item[(ii)] There exists a meromorphic continuation of $\mathcal{Q}(s)$ to (a branch cut of) an open neighborhood of ${\rm Re}(s) \ge 1/a$
    \item[(iii)] The rightmost singularity is at $s=1/a$ with order $b$ given by the average value of $b_p$, i.e.
    \[
    b:=\lim_{x\to \infty}\frac{1}{\pi(x)}\sum_{p\le x}b_p.
    \]
\end{enumerate}
Moreover, a branch cut is required in (ii) if and only if $b\not\in \Z$.
\end{heuristic}

A ``singularity of order $b$" at $s=1/a$ means that $(s-1/a)^b\mathcal{Q}(s)$ is holomorphic and nonzero at $s=1/a$. If $b\in \Z_{\le 0}$, we interpret this to mean that $\mathcal{Q}(s)$ has no singularities in an open neighborhood of ${\rm Re}(s) \ge 1/a$. If $b=0$, we would predict $\mathcal{Q}(1/a) \ne 0$, and if $b\in \Z_{<0}$ we would predict $\mathcal{Q}(1/a) = 0$ is a root of order $|b|$.

Example \ref{ex:intro} is an excellent example of this. The Euler product
\[
\prod_p (1 + 2p^{-s})
\]
has minimum degree term $p^{-s}$ (so $a=1$) with leading coefficient a constant $b=2$. This corresponds to the pole at $s=1$ of order $2$ inherited from the $\zeta(s)^2$ factor.

We are being purposefully vague with the term ``reasonable" in this heuristic. This is basically an educated guess, which can fail in the wrong conditions. Certainly the average value $b$ needs to exist, and moreover each Euler factor needs to converge.

The results we summarize and prove in this paper will imply the following useful result:
\begin{corollary}\label{cor:heur}
    Heuristic \ref{heur} is true if both of the following hold:
    \begin{enumerate}[(a)]
        \item The Euler factors $1 + b_pz^a + \cdots$ converge absolutely for all $|z|<1$ and are nonzero at $z=p^{-1/a}$, and
        \item The function $p\mapsto b_p$ is Frobenian, that is it factors through the Artin map of a finite extension $K/\Q$ for all but finitely many primes.
    \end{enumerate}
\end{corollary}

We will also see several examples beyond this corollary where the Factorization Method works.

\begin{remark}
    There are cases where the Factorization Method produces the rightmost singularity, but for which Heuristic \ref{heur} is false. One source of such an issue is suggested by Corollary \ref{cor:heur}(a) - any Euler factor which vanishes at $s=1/a$ will change the order of the singularity.
\end{remark}

\section{Absolute Convergence of Euler Products}
The Factorization Method shines by expanding the region of absolute convergence after factoring out copies of the Riemann zeta function. In Example \ref{ex:intro}, the Dirichlet series
\[
\prod_p (1 + 2p^{-s})
\]
converges absolutely on the region ${\rm Re}(s) > 1$, while
\[
\zeta(s)^{-2}\prod_p (1 + 2p^{-s}) = \prod_p ( 1-3p^{-2s}+2p^{-3s})
\]
converges absolutely on the region ${\rm Re}(s) > 1/2$.

Absolute convergence of an Euler product is readily checked by taking the logarithm, and checking absolute convergence of the resulting sum. However, rather than composing with a log function each time, we can use the following theorem as a shortcut:

\begin{theorem}\label{thm:abs_conv}
    Let $f:\Z^+\to \C$ be a multiplicative function. Then, for each $s\in \C$, the generating Dirichlet series $\sum f(n) n^{-s}$ converges absolutely if and only if the series on prime powers
    \[
    \sum_{p\text{ prime}}\sum_{k=1}^{\infty} f(p^k) p^{-ks}
    \]
    converges absolutely.
\end{theorem}

This is a special case of classical results for the absolute convergence of infinite products, some examples of which can be found in \cite[Chapter VII]{knopp1990}. Using this theorem, together with a convergence test from a typical Calculus course, we can check the absolute convergence claim in Example \ref{ex:intro}.

\begin{BOXexample}\label{ex:intro_abs_conv}
    We will prove that
    \[
    \prod_p (1 -3p^{-2s} + 2p^{-3s})
    \]
    converges absolutely on the region ${\rm Re}(s) > 1/2$, as claimed in Example \ref{ex:intro}. By Theorem \ref{thm:abs_conv}, this is equivalent to the convergence of the series
    \[
    \sum_p \left\lvert\frac{-3}{p^{-2s}}\right\rvert + \left\lvert\frac{2}{p^{-3s}}\right\rvert = \sum_p \frac{3}{p^{2{\rm Re}(s)}} + \frac{2}{p^{3{\rm Re}(s)}}.
    \]
    By the $p$-series convergence test, this series converges if and only if both $2{\rm Re}(s) > 1$ and $3{\rm Re}(s) > 1$, which is equivalent to ${\rm Re}(s) > 1/2$.
\end{BOXexample}

For nearly every Euler product we consider in this paper, the process of checking absolute convergence will be the same: apply Theorem \ref{thm:abs_conv}, then compare the sum over prime powers with a $p$-series to determine the region of absolute convergence. Sometimes other convergence tests, like the ratio test, will be useful as well.

In practice, Theorem \ref{thm:abs_conv} and a Calculus convergence test is typically sufficient for the cases of interest in arithmetic statistics. For this reason, we will leave the process of checking absolute convergence in subsequent examples as exercises for the reader.

We include a proof of Theorem \ref{thm:abs_conv} here for the sake of completeness, based on pre-existing proofs such as those found in \cite{knopp1990}.

\begin{proof}[Proof of Theorem \ref{thm:abs_conv}]
    Given that $f$ is multiplicative, the absolute Dirichlet series formally factors as an Euler product
    \[
    \sum_{n=1}^{\infty} |f(n)n^{-s}| = \prod_p\left(1 + \sum_{k=1}^{\infty}|f(p^k)p^{-s}|\right).
    \]

    Suppose first that the Dirichlet series $\sum f(n) n^{-s}$ converges absolutely. Then certainly the sum
    \[
    \sum_{p\text{ prime}}\sum_{k=1}^{\infty} |f(p^k) p^{-ks}| \le \sum_{n=1}^{\infty} |f(n)n^{-s}|
    \]
    converges absolutely by the Comparison test, as it is a sum over a subset of terms in the original series.

    Conversely, suppose the sum over prime powers converges absolutely. The partial sums of $\sum |f(n) n^{-s}|$ necessarily form an increasing sequence. It suffices to show that this sequence is also bounded, as any bounded and monotonic sequence is convergent. The partial sums are bounded by
    \begin{align*}
    \sum_{n=1}^N |f(n)n^{-s}| &\le \prod_{p\le N}\left(1 + \sum_{k\le \log_2(N)}|f(p^k)p^{-s}|\right)\\
    &\le \prod_{p\le N} \exp\left(\sum_{k\le \log_2(N)}|f(p^k)p^{-s}|\right)\\
    &=\exp\left(\sum_{p\le N}\sum_{k\le \log_2(N)}|f(p^k)p^{-s}|\right)\\
    &\le \exp\left(\sum_{p}\sum_{k=1}^{\infty}|f(p^k)p^{-s}|\right).
    \end{align*}
    Thus, absolute convergence of the sum over prime powers shows the partial sums of $\sum |f(n) n^{-s}|$ form a bounded sequence. This concludes the proof.
\end{proof}

\section{Examples with Constant Coefficients}\label{sec:examples_with_const_coef}
Following Example \ref{ex:intro_expanded}, whenever the coefficients of the Euler factors do not depend on $p$ we can perform the Factorization Method without modifying the argument. Often this will verify Heuristic \ref{heur}, however zeros and singularities of individual Euler products can cause the heuristic to fail.

Given a power series $Q(z) = 1 + bz^a + \cdots\in 1 + z\C[[z]]$ with constant coefficients, consider the corresponding Euler product
\[
\mathcal{Q}(s) = \prod_p Q(p^{-s}).
\]
We take $b p^{-as}$ to be the minimum degree term, matching the usual notions of ``degree" and ``term" from the language of power series. As described in Example \ref{ex:intro_expanded}, we can make a strategic choice for Euler factor by the dictionary
\[
b p^{-as} \longmapsto (1- p^{-as})^{b}.
\]

Wood proves that the Factorization method with this choice of strategic Euler factor will always work when the coefficients do not depend on $p$ in \cite[Lemma 2.12]{wood2009}, although these cases were a standard result at the time. Estermann proved the first result of this kind when $h(z)$ is a polynomial \cite{estermann_1928}, while Dahlquist dealt with the general cases that $h(z)$ is holomorphic \cite{dahlquist_1952}. When one takes care to keep track of the orders of the singularities in Estermann and Dahlquist's proofs, their methods reproduce the rightmost singularity achieved by Wood. We discuss the statements and give self-contained proofs of Estermann and Dahlquist's results in Part \ref{part:existence}.

\begin{remark}
    Notice that we specifically claimed that Wood's, Estermann's, and Dahlquist's results imply that the Factorization Method always works on Euler products with constant coefficients. We \emph{did not} claim that Heuristic \ref{heur} is always true in these cases. Frei--Loughran--Newton remark on the possibility that Heuristic \ref{heur} can fail in \cite[Remark 2.4]{frei-loughran-newton2019}. We will see an example of this heuristic failing in Example \ref{ex:bad_euler_factor}.
\end{remark}

Let's apply the Factorization Method to a new example, slightly more complicated than Example \ref{ex:intro_expanded}.

\begin{BOXexample}\label{ex:complex_pole}
Consider the Euler product
\[
\mathcal{Q}(s) = \prod_p (1 + (1+i)p^{-2s} - p^{-3s}).
\]
We now perform the four steps of the Factorization Method.
\begin{enumerate}[(1)]
    \item The lowest degree term is given by $(1+i)p^{-2s}$ so that $a=2$ and $b=1+i$.
\end{enumerate}
Using Heuristic \ref{heur}, we can now predict that the rightmost singularity is at $s=1/2$ of order $1+i$. This singularity would have non-integral order! This means that we predict that $(s-1/2)^{1+i}\mathcal{Q}(s)$ is holomorphic on a neighborhood of ${\rm Re}(s) \ge 1/2$. In particular, we expect that we can only achieve meromorphic continuation of $\mathcal{Q}(s)$ on a branch cut. However, as we shall see, the rest of the argument is unaffected.

\begin{enumerate}
    \item[(2)] We multiply the top and bottom by our strategic factor $(1-p^{-as})^b = (1-p^{-2s})^{1+i}$.
    \item[(3)] Pulling the denominators out front, we decompose the Euler product as
    \[
    \mathcal{Q}(s) = \zeta(2s)^{1+i}\prod_p (1 + (1+i)p^{-2s} - p^{-3s})(1-p^{-2s})^{1+i}.
    \]
    \item [(4)] In order to distribute the Euler factors, we must expand $(1-p^{-2s})^{1+i}$ according to the binomial Taylor series. For the purposes of this example, we only need to first few terms
    \begin{align*}
        \mathcal{Q}(s) &= \zeta(2s)^{1+i}\prod_p (1 + (1+i)p^{-2s} - p^{-3s})\left(1 - (1+i)p^{-2s} + O(p^{-4s})\right)\\
        &= \zeta(2s)^{1+i}\prod_p \left(1 - p^{-3s} + O(p^{-4s})\right).
    \end{align*}
    Among ${\rm Re}(s)>0$, the implied constant in the big-oh can be bounded independent of $p$ using the remainder theorem for Taylor series. We leave it as an exercise to the reader to check that this Euler product converges absolutely on the region ${\rm Re}(s) > 1/3$.
\end{enumerate}

Thus, $\mathcal{Q}(s)$ inherits the singularity of $\zeta(2s)^{1+i}$ at $s=1/2$ of order $1+i$, verifying Heuristic \ref{heur} for this Euler product.
\end{BOXexample}

\begin{remark}
    We need to be careful with functions like $(s-1/2)^{1+i}$ or $\zeta(2s)^{1+i}$. Complex exponential functions, like the square root function, are technically multi-valued complex functions. This is why branch cuts are required. Throughout this paper, we take the convention
    \[
    z^\alpha = e^{\alpha \log(z)},
    \]
    where $\log(z)$ is some branch cut of the log function. You may think of this as the principal branch cut if you like, cutting the complex plane along the negative real axis.

    In practice, it is convenient in arithmetic statistics to only work with branch cuts that cut a line to the left of a singularity. This is because of how the Selberg--Delange method works by moving contour integrals from right to left. This is always possible, as the only thing a branch cut needs to do is prevent one from drawing a loop around the singularity.
\end{remark}

We can widen the families of Euler products we consider by allowing finitely many primes to ``break the pattern" established by the other Euler factors. These exceptional factors are merely pushed to the side, to be consider at the end. This happens fairly often in arithmetic statistics, where certain wild ramification behavior can induce different Euler factors.

\begin{BOXexample}
The generating series of absolute quadratic discriminants is given by
\[
(1 + 4^{-s} + 2\cdot 8^{-s})\prod_{p\ne 2}(1 + p^{-s}).
\]
The way that $2$ divides a quadratic discriminant is more varied than for odd primes, which comes from the wild ramification type of $2$ in quadratic extensions.

We can shove the factor at $2$ to the side, and complete to remaining product with a better factor obeying the established pattern:
\[
\frac{1 + 4^{-s} + 2\cdot 8^{-s}}{1+2^{-s}}\prod_{p}(1 + p^{-s}) = \frac{1 + 4^{-s} + 2\cdot 8^{-s}}{1+2^{-s}}\frac{\zeta(s)}{\zeta(2s)}.
\]
Thus, the rightmost pole is at $s=1$ of order $1$, agreeing with the average value of
\[
b_p = \begin{cases}
0 & p=2\\
1 & p\ne 2.
\end{cases}
\]
\end{BOXexample}

It is often the case that number theoretic facts hold for ``all but finitely many primes". This example shows that finitely many primes are no obstacle for meromorphic continuation.

We caution the reader that bad primes may not be obvious at first glance, and that they can break Heuristic \ref{heur} as we see in the next example.

\begin{BOXexample}\label{ex:bad_euler_factor}
    Consider the Euler product
    \begin{align*}
    \mathcal{Q}(s) &= \prod_p \frac{1}{1 - 2p^{-s}}\\
    &=\prod_p(1 + 2p^{-s} + 4p^{-2s} + \cdots ).
    \end{align*}
    It does not look like we have any bad primes, as all the Euler factors look the same. Let's try out the Factorization Method.
    \begin{enumerate}[(1)]
        \item The lowest degree term is $2p^{-s}$ of degree $a=1$ with constant coefficient $b=2$.
    \end{enumerate}
    Thus, Heuristic \ref{heur} predicts the rightmost pole is at $s=1$ of order $2$. We proceed with the next steps of the method.
    \begin{enumerate}
        \item[(2)] We multiply the top and bottom by the strategic Euler factor $(1 - p^{-s})^{2}$ to produce
        \begin{align*}
            \mathcal{Q}(s) &= \prod_p \frac{1}{(1-p^{-s})^2} \prod_p (1+2p^{-s}+4p^{-2s}+\cdots)(1-p^{-s})^2.
        \end{align*}
        \item[(3)] We distribute the Euler factors
        \begin{align*}
            \mathcal{Q}(s) &=\zeta(s)^2 \prod_p (1+2p^{-s}+4p^{-2s}+\cdots)(1-2p^{-s}+p^{-2s})\\
            &=\zeta(s)^2 \prod_p (1+p^{-2s}+2p^{-3s}+4p^{-4s}+\cdots),
        \end{align*}
        where we see that the minimum degree terms at $p^{-s}$ do cancel out.
    \end{enumerate}

    It is tempting to skip part of step (4) in the Factorization method and say it is left as an exercise to the reader\footnote{Just like the lazy author does in the many examples in this paper.} to verify that the remaining Euler product converges absolutely on the region ${\rm Re}(s)>1/2$, because the remaining term of smallest degree is $p^{-2s}$.

    Except this is not true! The remaining Euler factors are geometric series
    \[
    1 + \sum_{k=0}^\infty 2^{k}p^{-(k+2)s} = 1 + p^{-2s}\sum_{k=0}^\infty (2p^{-s})^k,
    \]
    which converge absolutely if and only if $|2p^{-s}| < 1$. This is equivalent to ${\rm Re}(s) > \frac{\log 2}{\log p}$. By taking the sum over prime powers as described in Theorem \ref{thm:abs_conv}, we are forced to conclude that the remaining Euler product converges absolutely on the region ${\rm Re}(s) > 1$ (thanks to the $p=2$ factor). This region is not big enough to contain the predicted singularity at $s=1$. Thus, the Factorization Method failed at step $(4)$.

    Let us try again, but this time treat the $p=2$ factor as exceptional and pull it to the outside first. Steps (1) and (2) are the same for the Euler product over $p\ne 2$, so that our factorization in step (3) simplifies to:
    \begin{align*}
        \mathcal{Q}(s) &= \frac{1}{1-2^{1-s}}\prod_{p\ne 2}\frac{1}{(1-p^{-s})^2} \prod_{p\ne 2} (1 + p^{-2s} + 2p^{-3s} + 4p^{-4s}+\cdots)\\
        &= \frac{(1-2^{-s})^2}{1-2^{1-s}}\zeta(s)^2 \prod_{p\ne 2} (1 + p^{-2s} + 2p^{-3s} + 4p^{-4s}+\cdots).
    \end{align*}
    However, without the $p=2$ factor to mess things up, the remaining Euler product
    \[
    \prod_{p\ne 2} (1 + p^{-2s} + 2p^{-3s} + 4p^{-4s}+\cdots) = \prod_{p\ne 2}\left(1 + p^{-2s}\sum_{k=0}^\infty (2p^{-s})^k\right)
    \]
    is absolutely convergent for ${\rm Re}(s) > \frac{\log 2}{\log 3}$. This region is now large enough to contain the singularity at $s=1$, concluding step (4).

    We can then conclude that $\mathcal{Q}(s)$ has a pole at $s=1$ of order $3$, contradicting Heuristic \ref{heur}. This pole comes from the pole of order $2$ at $s=1$ in $\zeta(s)^2$ and the simple pole at $s=1$ of the exceptional factor $\frac{(1-2^{-s})^2}{1-2^{1-s}}$.

    In fact, this is not the only rightmost pole. The factor $\frac{(1-2^{-s})^2}{1-2^{1-s}}$ has simple poles at $s = 1 + \frac{2\pi k}{\log 2} i$ for each $k\in \Z$, which all lie on the line ${\rm Re}(s) = 1$.
\end{BOXexample}

This example is a good justification for having the notion of ``reasonable" in Heuristic \ref{heur} require that the individual Euler factors $1+b_p p^{-s} + \cdots$ converge and are nonzero on an open neighborhood of ${\rm Re}(s) \ge 1/a$.

It is notable that Dahlquist \cite{dahlquist_1952} makes no such restriction in their work. Indeed, the only problem that $p=2$ caused in our example was that it threw off our prediction for the order of the pole. The Euler product still has a meromorphic continuation, the existence of which was Dahlquist's primary goal. By adjusting appropriately for Euler factors with ``bad" behavior, this behavior can be accounted for.

\section{Examples with Frobenian Coefficients}\label{sec:examples_with_Frob_coef}

Another important case for arithmetic statistics is when $b_p$ is determined by the splitting type of $p$ in a finite extension of $\Q$. Classical examples of this property are Dirichlet characters. For example, the Dirichlet character
\[
\chi(n) = \begin{cases}
    1 & n\equiv 1 \pmod 4\\
    -1 & n\equiv -1\pmod 4\\
    0 & 2\mid n
\end{cases}
\]
has $\chi(p)$ determined by the splitting type of $p$ in $\Q(i)/\Q$.

\begin{definition}\label{def:frobenian}
    Let $K$ be a number field and $A$ a set. A function $f:\{\text{primes of }K\} \to A$ is called \textbf{Frobenian} if there exists a finite extension $E/K$, a finite set of places $S$, and a class function\footnote{Recall that a class function is a function which is constant on conjugacy classes.} $c:\Gal(E/K) \to A$ such that the diagram
    \[
    \begin{tikzcd}
        \{\text{primes of }K\} - S \rar{f}\dar[swap]{\text{Artin map}} &A\\
        \Gal(E/K) \urar[swap]{c}
    \end{tikzcd}
    \]
    commutes.
\end{definition}

To our best knowledge, the term Frobenian function is due to Serre. A nice overview of the concept can be found in \cite[Subsection 3.3]{serre2012} as well as \cite[Section 2]{loughran-matthiesen2019}.

We call an Euler product $\prod Q_p(p^{-s})$ \textbf{Frobenian} if the function $p\mapsto Q_p(z)$ is Frobenian. To the author's knowledge, these Euler products were first studied by Kurokawa \cite{kurokawa_1978,kurokawa_1986I,kurokawa_1986II} and expanded on by Moroz \cite{moroz_1988} to cover all Frobenian polynomials (that is, when $Q_p(z)$ is a polynomial). We state and give a self-contained proof of Kurokawa and Moroz's results in Part \ref{part:existence}. These Euler products have received more attention in arithmetic statistics lately, as they arise naturally in the generating Dirichlet series of $G$-extensions. The meromorphic continuation of some special cases of this family are re-worked out as analytic lemmas of several more recent papers, including \cite[Proposition 2]{kaplan-marcinek-takloo-bighash2015} and \cite[Proposition 2.3]{frei-loughran-newton2019}, while the meromorphic continuation of any Frobenian Euler product with polynomial factors past the rightmost singularity is reproven in \cite[Proposition 2.2]{alberts2021}. All of these proofs are done with a version of the Factorization Method, although often written with different goals in mind.

We need to be more specific as to what we call a ``term" for Frobenian Euler products. We say that a term is something of the form $b \chi(\Fr_p)p^{-as}$ for $b$ a constant and $\chi$ a character of $E/K$. The degree of such a term is $a$. The irreducible characters of $\Gal(E/K)$ form a basis for the space of class functions, so that we can always decompose a Frobenian function $p\mapsto Q_p(z)$ into a sum of ``terms" of the form $b\chi(\Fr_p)p^{-as}$.

If $\chi$ is a Dirichlet character, we then choose a strategic Euler factor via the dictionary
\[
b \chi(p)p^{-as} \longmapsto (1 - \chi(p)p^{-as})^b.
\]
These are the Euler factors of Dirichlet $L$-functions, so we might expect $L$-functions to show up in step (3) of the Factorization method.

\begin{BOXexample}\label{ex:dirichlet_L_factor}
Consider the Euler product
\[
\mathcal{Q}(s) = \prod_{p\equiv 1 \mod 4} (1 + 4p^{-s}).
\]
The lowest degree terms are at $p^{-s}$ with $a=1$ and the coefficients $b_p$ are given by
\[
\begin{cases}
4 & p\equiv 1 \mod 4\\
0 & p\equiv 2,3 \mod 4,
\end{cases}
\]
with an average value $b = 2$. Thus, Heuristic \ref{heur} predicts the rightmost singularity to be a pole at $s=1$ of order $2$.

In order to perform the Factorization Method, we decompose $b_p$ into a linear combination of irreducible characters. Let $\chi_4$ be the nontrivial mod $4$ Dirichlet character, so that
\[
\mathcal{Q}(s) = \prod_{p\ne 2} (1 + 2p^{-s} +2\chi_4(p)p^{-s}).
\]
\begin{enumerate}[(1)]
    \item We treat \emph{both} $2p^{-s}$ and $2\chi_4(p)p^{-s}$ as minimum degree terms.
    \item We then multiply top and bottom by \emph{two} strategic Euler factors, both $(1-p^{-s})^{2}$ and $(1 - \chi_4(p)p^{-s})^{2}$.
    \begin{align*}
        \mathcal{Q}(s) &= \prod_{p\ne 2} \frac{(1 + 2p^{-s} +2\chi_4(p)p^{-s})(1-p^{-s})^2(1-\chi_4(p)p^{-s})^2}{(1-p^{-s})^2(1-\chi_4(p)p^{-s})^2}
    \end{align*}
    \item These Euler factors are related to the Riemann zeta function \emph{and} a Dirichlet $L$-function, so that pulling out the denominators gives the decomposition
    \begin{align*}
        \mathcal{Q}(s) 
        &=(1-2^{-2s})^2 \zeta(s)^2L(s,\chi_4)^2\prod_{p\ne 2} (1 + 2p^{-s} +2\chi_4(p)p^{-s})(1-p^{-s})^2(1-\chi_4(p)p^{-s})^2
    \end{align*}
    \item Distributing the remaining Euler factors gives
    \begin{align*}
        \mathcal{Q}(s)&=(1-2^{-2s})^2 \zeta(s)^2L(s,\chi_4)^2\prod_{p\ne 2}(1 + (2\chi_4(p)-2)p^{-2s} + (2 + \chi_4(p))p^{-3s}).
    \end{align*}
    Considering that $|\chi_4(p)|\le 1$, we leave it to the reader to check that the remaining Euler product is absolutely convergent on ${\rm Re}(s) > 1/2$.
\end{enumerate}

The fact that $L(1,\chi_4)\ne 0$ implies that $\mathcal{Q}(s)$ inherits its rightmost singularity solely from $\zeta(s)^2$, which is a pole at $s=1$ of order $2$ as predicted.
\end{BOXexample}

As the $b_p$ become more complicated, the behavior of $L$-functions at $s=1$ becomes more important. Luckily, we know that all Dirichlet $L$-functions are holomorphic and nonzero at $s=1$. For arbitrary Frobenian Euler products, we may have characters $\chi$ that are not Dirichlet characters. Suppose $\chi = {\rm tr}\rho$ for the Galois representation $\rho:\Gal(E/K)\to {\rm GL}(V)$. We then use the dictionary
\[
b \chi(\Fr_p) p^{-as} \longmapsto \det(I - \rho(\Fr_p|V^{I_p})p^{-as})^b
\]
to choose our strategic Euler factors. You may recognize these as (powers of) the Euler factors of Artin $L$-functions, so we expect some Artin $L$-functions to arise in step (3). Let us see one more example, using an irreducible character of dimension $>1$.

\begin{BOXexample}
    Let $\rho:G_\Q \to {\rm GL}_2(\F_2)$ be a Galois representation coming from the Galois action on $2$-torsion of the elliptic curve $E:y^2 = x^3 + x + 1$. Consider the Euler product
    \[
    \mathcal{Q}(s) = \prod_p \left(1 + {\rm tr}\rho(\Fr_p|E[2]^{I_p})p^{-2s} + 2p^{-3s}\right).
    \]
    The $2$-torsion points are precisely the point at infinity together with $(\alpha,0)$ for each root $\alpha$ of $x^3 + x + 1$. Thus, the fixed field of $\ker\rho$ is the splitting field $E$ of $x^3+x+1$ over $\Q$, which has $\Gal(E/\Q)\cong S_3 \cong {\rm GL}_2(\F_2)$. It follows that $\mathcal{Q}(s)$ is Frobenian, as the action of $\Fr_p$ on $E[2]^{I_p}$ is determined by the splitting type of $p$ in $E/\Q$. In particular, this action factors through the irreducible two-dimensional representation of $S_3$, so we can conclude that $\rho$ is irreducible.

    \begin{enumerate}[(1)]
        \item The minumum degree term is ${\rm tr}\rho(\Fr_p|E[2]^{I_p}) p^{-2s}$, which has degree $a=2$ and average value
        \[
            b = \sum_{g\in {\rm GL_2(\F_2)}} {\rm tr}\rho(g) = 0
        \]
        by $\rho$ irreducible. Heuristic \ref{heur} predicts a pole at $s=1/2$ of ``order $0$". A pole of ``order $0$" is no pole at all, so Heuristic \ref{heur} is really predicting an analytic continuation to an open neighborhood of ${\rm Re}(s) \ge 1/2$ with $\mathcal{Q}(1/2)\ne 0$.

        \item We choose our strategic Euler factor
        \[
        \det(I - \rho(\Fr_p|E[2]^{I_p}) p^{-2s}) = 1 - {\rm tr}\rho(\Fr_p|E[2]^{I_p}) p^{-2s} + \det\rho(\Fr_p|E[2]^{I_p})p^{-4s}.
        \]
        \item We can then compute
        \begin{align*}
            \mathcal{Q}(s) =& \prod_p \det(I - \rho(\Fr_p|E[2]^{I_p}) p^{-2s})^{-1}\\
            &\times \prod_p\left(1 + {\rm tr}\rho(\Fr_p|E[2]^{I_p})p^{-2s} + 2p^{-3s}\right)\det(I - \rho(\Fr_p|E[2]^{I_p}) p^{-2s})\\
            =& L(2s,\rho)\prod_p\left(1 + 2p^{-3s} +O(p^{-4s})\right).
        \end{align*}
        \item The Artin $L$-function $L(2s,\rho)$ is known to be holomorphic and nonzero on an open neighborhood of ${\rm Re}(s)\ge 1/2$, while the remaining Euler product converges absolutely on the region ${\rm Re}(s) > 1/3$.
    \end{enumerate}
    This verifies Hueristic \ref{cor:heur} by giving a holomorphic continuation of $\mathcal{Q}(s)$ to the region ${\rm Re}(s) > 1/3$, in particular showing that $\mathcal{Q}(s)$ has a removable singularity at $s=1/2$ with $\mathcal{Q}(1/2)\ne 0$.
\end{BOXexample}

In general, Frobenian Euler products will involve Artin $L$-functions in their factorizations. This is useful because we know more about the analytic properties of Artin $L$-functions than for other Dirichlet series. The reader may find the following facts useful:
\begin{itemize}
    \item Artin $L$-functions converge absolutely on the region ${\rm Re}(s) > 1$.
    \item If $\rho\ne 1$ is irreducible, then $L(s,\rho)$ has an analytic continuation to an open neighborhood of ${\rm Re}(s) \ge 1$ with $L(1,\rho)\ne 0$.
    \item Artin $L$-functions have a meromorphic continuation to the entire complex plane, following from Brauer's theorem on induced characters and the fact below on Hecke $L$-functions.
    \item Dirichlet (respectively Hecke) $L$-functions for nontrivial irreducible Dirichlet (respectively Hecke) characters are entire functions.
\end{itemize}
The reader may consult an analytic number theory textbook, such as \cite{iwaniec-kowalski2004}, for further information.

The \textbf{Artin Conjecture} predicts that $L(s,\rho)$ is entire for any nontrivial irreducible representation $\rho$, however this is currently unproven. For extra clean results on meromorphic continuations one may wish to assume the Artin conjecture, although it will not be necessary for the Factorization Method.

\section{Further Applications of the Factorization Method}

In this section we consider three other variations of Euler products that the reader may encounter. The Factorization Method can be applied to all of these types of Euler products, and many more beyond what we consider in this paper.

\subsection{Euler Products over Number Fields}

Euler products can also be considered over primes of fields other than $\Q$. In this case, the Riemann zeta function can be replaced with the Dedekind zeta function, Dirichlet $L$-functions with Hecke $L$-functions, and so on.

In the case of Frobenian Euler products, this will always work out because of the following:
\begin{theorem}\label{thm:frob_K_to_Q}
Let $K/\Q$ be a finite extension. If $\mathfrak{p}\mapsto H_\mathfrak{p}(z)$ is Frobenian over $K$, then
\[
p\mapsto h_p(z) := \prod_{\mathfrak{p}\mid p} H_{\mathfrak{p}}(z^{f_{\mathfrak{p}}})
\]
is Frobenian over $\Q$.
\end{theorem}
By shifting perspective back down to $\Q$, we can re-use meromorphic continuation results, such as Moroz's results \cite{moroz_1988}, for Euler products over rational primes. From a theoretical perspective, there is no distinction.

From a practical perspective, it is often convenient to take the Euler product over the largest field possible. This is because the Dedekind zeta functions group together a lot of disparate $L$-functions. This can be used to state decompositions more concisely, and also as a shortcut for confirming holomorphicity on larger regions without needing to assume the Artin Conjecture.

\begin{BOXexample}
Consider the Euler product
\[
\prod_{\mathfrak{p}} (1 + 2|\mathfrak{p}|^{-s})
\]
over the finite places of $\Q(\sqrt{-3})/\Q$. Here, $|\mathfrak{p}|$ denotes the norm of $\mathfrak{p}$ down to $\Q$.

On the one hand, this decomposes into a Frobenian Euler product over $\Q$
\[
\prod_{p\equiv 0,1 \mod 3} (1 + 2p^{-s})^2\prod_{p\equiv 2\mod 3} (1 + 2p^{-2s}).
\]
The reader can confirm that the rightmost pole is at $s=1$ of order $2$, verifying Heuristic \ref{heur}, by using the strategic Euler factors from Dirichlet $L$-functions as in Example \ref{ex:dirichlet_L_factor}.

On the other hand, we can also factor the Euler product by treating it as an Euler product over primes of $\Q(\sqrt{-3})$ with constant coefficients.
\begin{enumerate}[(1)]
    \item The minimum degree term is then $2|\mathfrak{p}|^{-s}$.
    \item We multiply top and bottom by the strategic factor $(1-|\mathfrak{p}|^{-s})^2$ in analogy with Example \ref{ex:intro_expanded}.
    \item This produces a decomposition
    \[
    \zeta_{\Q(\sqrt{-3})}(s)^2 \prod_{\mathfrak{p}}(1 - 3|\mathfrak{p}|^{-2s} + 2|\mathfrak{p}|^{-3s}).
    \]
    \item The remaining Euler product converges absolutely on the region ${\rm Re}(s) > 1/2$ by an argument analogous to Example \ref{ex:intro_abs_conv}.
\end{enumerate}
Thus, we see the pole at $s=1$ of order $2$ was inherited from the Dedekind zeta function $\zeta_{\Q(\sqrt{-3})}(s)^2$.
\end{BOXexample}

\begin{proof}[Proof of Theorem \ref{thm:frob_K_to_Q}]
Let $E/K$ be the finite extension for which $\mathfrak{p}\mapsto H_{\mathfrak{p}}(z)$ factors through the map $\mathfrak{p} \mapsto \Leg{E/K}{\mathfrak{p}}$.

We first check that the Euler products are even equal. Let $|\mathfrak{p}|$ denote that norm of $\mathfrak{p}$ down to $\Q$, so that
\begin{align*}
    \prod_p h_p(p^{-s}) &= \prod_p \prod_{\mathfrak{p}\mid p}H_\mathfrak{p}(p^{-f_{\mathfrak{p}}s})\\
    &= \prod_p \prod_{\mathfrak{p}\mid p}H_\mathfrak{p}(|\mathfrak{p}|^{-s})\\
    &=\prod_{\mathfrak{p}}H_\mathfrak{p}(|\mathfrak{p}|^{-s}).
\end{align*}
Next, we confirm that $p\mapsto h_p(z)$ is Frobenian. Indeed, consider that $\Leg{E/\Q}{p}$ determines the sequence of pairs $f_{\mathfrak{p}}, \Leg{E/K}{\mathfrak{p}}$. The inertia degree is given by the order of the image of $\Leg{E/\Q}{p}$ in $\Gal(K/\Q)$. Meanwhile, $\Leg{E/K}{\mathfrak{p}} = \Leg{E/\Q}{p}^{f_{\mathfrak{p}}}$. As these pairs determine $h_p(z)$, we have proven it is Frobenian over $E/\Q$.
\end{proof}

\subsection{An Example with Growing Coefficients}
While our focus for this paper is on Euler products with constant or Frobenian coefficients, the Factorization Method will work for more general Euler products. These, however, need not be ``reasonable" in the context of Heuristic \ref{heur}, and the heuristic will often fail.

\begin{BOXexample}\label{ex:growing_coef}
Consider the Dirichlet series
\[
\mathcal{Q}(s) = \prod_p (1 + p^{-s} + (p+1)p^{-2s}).
\]
The lowest degree term appears to be $p^{-s}$, so we can take $a=1$ and $b=1$. Thus, Heuristic \ref{heur} suggests that its rightmost singularity is a simple pole at $s=1$.

However, we can demonstrate that this is not true. Performing steps (2) and (3) of the Factorization Method by multiplying top and bottom by $1-p^{-s}$, we factor this series as
\[
\mathcal{Q}(s) = \zeta(s) \prod_p (1 + p^{1-2s} - (p+1)p^{-3s}).
\]
However, the remaining Euler product still diverges at $s=1$. We repeat this process for another strategic Euler factor $1-p^{1-2s}$, which implies
\[
\mathcal{Q}(s) = \zeta(s)\zeta(2s-1)\prod_p (1 - (p+1)p^{-3s}- p^{2-4s} - (p+1)p^{1-5s}).
\]
The right hand Euler product converges absolutely for ${\rm Re}(s) > 3/4$. Meanwhile, \emph{both} $\zeta(s)$ and $\zeta(2s-1)$ have a simple pole at $s=1$, implying that $\mathcal{Q}(s)$ has a pole of order $2$ at $s=1$.
\end{BOXexample}

It is certainly possible to generalize Heuristic \ref{heur} to include Euler products whose coefficients are polynomial in $p$, like Example \ref{ex:growing_coef}, where we might require that both $p^{-s}$ and $p^{1-2s}$ be considered ``minimum degree terms". We would need to give a more general description of what we mean by a ``term", how to define the degree of a term, and how to choose a corresponding strategic Euler factor. However, such Euler products can be much more complicated. See Section \ref{sec:natural_boundary} for some references and a brief discussion on the natural boundary of such Euler products, which is still not known in every case.

\subsection{Additive Factorizations}

While the Factorization Method has primarily been described with multiplicative language, a lot of structure can be seen by first taken the logarithm and then decomposing the resulting infinite sums.

\begin{BOXexample}
Consider the Euler product
\[
\mathcal{Q}(s) = \prod_p \exp(b_pp^{-s}) = \prod_p \left( 1 + b_p p^{-s} + \frac{b_p}{2}p^{-2s} + \cdots \right).
\]
We claim that if the average value of $b_p$ exists \emph{with a power savings}, then Heuristic \ref{heur} is true for this Euler product. By a power savings, we mean that
\[
\sum_{p\le x} b_p = b\pi(x) + O(x^{1-\delta})
\]
for some $\delta > 0$.

Rather than multiplying top and bottom by some strategic factor, we take the log
\begin{align*}
\log\mathcal{Q}(s) &= \sum_p b_p p^{-s},
\end{align*}
then add and subtract a strategic factor:
\begin{align*}
    \sum_p bp^{-s} + \sum_p (b_p-b)p^{-s}
\end{align*}
The first piece will contribute the singularity as in step (3). Consider that
\[
\exp\left(\sum_p b p^{-s}\right) = \prod_p \exp(bp^{-s}) = \prod_p\left(1 + bp^{-s} + \frac{b^2}{2}p^{-2s} + \cdots\right).
\]
Applying the Factorization Method by multiplying top and bottom by $(1-p^{-s})^b$, we can factor this series as $\zeta(s)^b$ times an Euler product absolutely convergent on ${\rm Re}(s) > 1/2$. We leave the details as an exercise to the reader, as they are similar to the previously considered examples with constant coefficients. Thus, this term contributes a singularity at $s=1$ of order $b$.

The second piece can be bounded as in step (4) using the power savings. We rewrite the series as an integral using a trick classically applied to the Riemann zeta function (called a Mellin transform):
\begin{align*}
    \left\lvert\sum_{p} (b_p -b)p^{-s}\right\rvert &= \left\lvert\sum_{p} (b_p -b)(-s)\int_p^{\infty} x^{-s-1}\ dx\right\rvert\\
    &=\left\lvert s\int_1^\infty \left(\sum_{p\le x} (b_p - b)\right)x^{-s-1}\ dx\right\rvert\\
    &\le|s|\int_1^\infty x^{1-\delta}x^{-{\rm Re}(s)-1}\ dx.
\end{align*}
This integral converges if and only if $1-\delta -{\rm Re}(s)-1 < -1$, or equivalently ${\rm Re}(s) > 1-\delta$. By the comparison theorem, the second piece converges in this region. Consequently, the series
\[
\exp\left( \sum_p (b_p - b) p^{-s}\right) = \prod_p \exp((b_p - b)p^{-s})
\]
is holomorphic \emph{and nonzero} on the region ${\rm Re}(s)>1-\delta$, as the exponential function has no zeros.

Thus, $\mathcal{Q}(s)$ has an analytic continuation to a branch cut of an open neighborhood of ${\rm Re}(s) \ge 1$, with a singularity at $s=1$ of order $b$.
\end{BOXexample}

The use of Mellin transforms is an important trick in arithmetic statistics, and plays a major role in the Selberg--Delange method. The reader may find the general formula useful:
\[
n^{-s} = -s\int_n^{\infty} x^{-s-1}\ dx,
\]
where the integral converges absolutely on the region ${\rm Re}(s) >0$.

\subsection{Further References} It is not possible to adequately treat every useful case of the factorization method. Generally speaking, anytime you have a meromorphic function with an Euler product, those Euler factors can be used as ``strategic factors" in step (2) of the Factorization Method.

We acknowledge a couple more scenarios below where this happens.

\begin{itemize}
    \item Any reasonable $L$-function in the sense of \cite[Section 5.1]{iwaniec-kowalski2004} has a meromorphic continuation and Euler product, and so can be used to give strategic Euler factors in step (2). These include many arithmetically $L$-functions discussed above such as Dirichlet $L$-functions, Hecke $L$-functions, and (conjecturally) Artin $L$-functions, but also include automorphic $L$-functions and the $L$-functions attached to varieties (conjecturally automorphic).

    \item Any multiple Dirichlet series, that is a Dirichlet series of more than one complex variable, which has an Euler product can also be used to give strategic factors in step (2). Generally we would want to consider different variables as having equal weight in step (1), so an Euler product $\prod_p (1+p^{-s} + p^{-w})$ has minimum degree terms $p^{-s}$ and $p^{-w}$. For some examples, see \cite{bhowmik_essouabri_lichtin_2007,delabarre_2013,delabarre2014}.
\end{itemize}

\section{Iterating the Factorization Method}\label{sec:iterating_factorization_method}

The rightmost singularity predicted by Heuristic \ref{heur} is not the end of the Factorization Method's usefulness. It may have occurred to the reader that the remaining Euler product could be analytically continued \emph{again} by repeating the Factorization Method.

\begin{BOXexample}
    Let's revist the Euler product
    \[
    \prod_p (1 + 2p^{-s}) = \zeta(s)^2 \prod_p(1 - 3p^{-2s} + 2p^{-3s})
    \]
    from Example \ref{ex:intro}. We previously showed that this function has a meromorphic continuation to the region ${\rm Re}(s) > 1/2$ with a pole at $s=1$ of order $2$.

    We can beat that result by applying the Factorization Method to the remaining Euler product. By multiplying top and bottom by $(1 - p^{-2s})^{-3}$, we compute
    \begin{align*}
        \prod_p (1 + 2p^{-s}) &= \zeta(s)^2\zeta(2s)^{-3}\prod_p\left(1 + O(p^{-3s})\right).
    \end{align*}
    Theorem \ref{thm:abs_conv} implies that the remaining Euler product converges absolutely on the region ${\rm Re}(s) > 1/3$.

    We have now shown that $\mathcal{Q}(s)$ has a meromorphic continuation to the larger region ${\rm Re}(s) > 1/3$ with a pole of order $2$ at $s=1$, and poles at any of the nontrivial zeroes of $\zeta(2s)$ that may land in this region. The Riemann Hypothesis predicts that the nontrivial zeroes of $\zeta(2s)$ all have real part $1/4 < 1/3$, but the Riemann Hypothesis is still unproven at this time.
\end{BOXexample}

Seeing this example, you might be thinking ``What if I did the Factorization Method again? and \emph{again}? and \textbf{\emph{again}}? Can we do better?" Congratulations, you just described the method used by Estermann \cite{estermann_1928}, Dahlquist \cite{dahlquist_1952}, and others \cite{kurokawa_1978,kurokawa_1986I,kurokawa_1986II,moroz_1988} to meromorphically continue Euler products to their natural boundary!

Indeed, it turns out that the Factorization Method can be applied countably many times to produce a \textbf{zeta product}. For the Euler product in Example \ref{ex:intro}, this looks like
\[
\prod_p (1 + 2p^{-s}) = \prod_{n=1}^{\infty} \zeta(ns)^{b_n}
\]
for some sequence of numbers $b_n$. The first example of such a zeta product appears in \cite{landau_walfisz_1920} for the Dirichlet series $\exp(\sum p^{-s})$. What is remarkable is that the tail of the zeta product
\[
\prod_{n\ge N} \zeta(ns)^{b_n}
\]
converges absolutely on the region ${\rm Re}(s) > 1/N$. Thus, taking a limit as the tail gets smaller, the zeta product produces a continuation to the region ${\rm Re}(s) > 0$ away from a set of isolated singularities which occur either at $s=1/n$ or $s=$ a zero of $\zeta(ns)$. These singularities are poles precisely when the corresponding number $b_n$ is a integer.

Part \ref{part:existence} of this paper concerns the existence of this type of zeta product, using it to produce a meromorphic continuation. Part \ref{part:explicit} of this paper concerns the sequence $b_n$ of exponents, where we give explicit results so that the orders of any singularities or poles can be easily computed.

\part{Existence of Analytic Continuations}\label{part:existence}

\section{History of Analytic Continuations using the Factorization Method}\label{sec:history}

\subsection{Euler Products with Constant Coefficients}\label{subsec:histroy_const_coef}
The first general result applying the Factorization Method is due to Estermann \cite{estermann_1928}.

\begin{theorem}[{Estemann \cite{estermann_1928}}]\label{thm:estermann}
    Let $h(z) \in 1 + z\Z[z]$ be a polynomial with integer coefficients. Then the Euler product
    \[
    \prod_p h(p^{-s})^{-1}
    \]
    has a meromorphic continuation to the right half-plane ${\rm Re}(s)>0$.

    If the roots of $h(z)$ all have magnitude $1$ then this Euler product can be meromorphically continued to $\C$, otherwise ${\rm Re}(s) = 0$ is the natural boundary of this function.
\end{theorem}

The \textbf{natural boundary} is the boundary of the largest open set on which the function can be meromorphically continued. Estermann's proof of analytic continuation can be understood as iterating the Factorization Method infinitely many times to produce a zeta product
\[
    \prod_p h(p^{-s})^{-1} = \prod_{n=1}^\infty \zeta(ns)^{b_n}
\]
for some sequence $(b_n)$. All the roots of $h(z)$ having magnitude $1$ is equivalent to the sequence $(b_n)$ having at most finitely many nonzero terms, making this function a finite product of zeta functions and therefore meromorphic on $\C$. Otherwise, the sequence $(b_n)$ has infinitely many nonzero terms. In this case, Estermann then proves that every point on the ${\rm Re}(s)=0$ is the limit point of zeros or poles of $\zeta(ns)^{b_n}$ as $n$ varies from $1$ to $\infty$, making every point on this line an essential singularity.

Dahlquist \cite{dahlquist_1952} then expanded on Estermann's work to include non-polynomial Euler factors. The proof outline is essentially the same, where Dalhquist needs to separate out Euler factors that contribute singularities or zeros.

\begin{theorem}[{Dahlquist \cite{dahlquist_1952}}]\label{thm:dahlquist}
    Let $Q(z)$ be holomorphic on (a branch cut of) the complex plane, save for a set of isolated singularities, holomorphic at $z=0$ with $Q(0)=1$, and such that the closed unit ball does not contain a limit point of zeros of singularities of $Q(z)$. Then the Euler product
    \[
    \mathcal{Q}(s) = \prod_p Q(p^{-s})
    \]
    has an analytic continuation to (a branch cut of) the right halfplane ${\rm Re}(s) > 0$, save for a set of isolated singularities.

    If $Q(z)$ is not a polynomial whose roots all have magnitude $1$, then the line ${\rm Re}(s)=0$ is the natural boundary for $\mathcal{Q}(s)$ (up to branch cuts).
\end{theorem}

We are careful to avoid the term meromorphic here, as Dahlquist's work necessarily allows for singularities of non-integral order. Dahlquist does not mention branch cuts in their theorem statements, although we know that any singularity of non-integral order requires a branch cut that prevents one from drawing a loop around the singularity. This is distinct from Estermann's original work, in which their Euler product is truly meromorphic on the region ${\rm Re}(s) > 0$.

We give self-contained proofs of the analytic continuations in Estermann's Theorem \ref{thm:estermann} and Dahlquist's Theorem \ref{thm:dahlquist}. The existence of an analytic continuation to a branch cut of ${\rm Re}(s) > 0$ is proven in Theorem \ref{thm:existence_const_coef}. In the context of Estermann's theorem, it is a consequence of Theorem \ref{thm:explicit_singularities_const_coef} that all of the singularities are poles so that no branch cut is necessary. We summarize the main idea behind the proof that ${\rm Re}(s)=0$ is the natural boundary in Section \ref{sec:natural_boundary}, though we leave many of the details to the original papers.

\subsection{Frobenian Euler Products}\label{subsec:history_Frob_coef}
Results of similar strength are true for Frobenian Euler products. Kurokawa \cite{kurokawa_1978,kurokawa_1986I,kurokawa_1986II} and Moroz \cite{moroz_1988} prove an analogous statement to Estermann's result, with Moroz's proof being a streamlined version of Kurokawa's work in \cite{kurokawa_1978}.

\begin{theorem}[{Kuokawa \cite{kurokawa_1978}, Moroz \cite{moroz_1988}}]\label{thm:km}
    Let $h_p(z) \in 1 + z R(\Gal(E/\Q))[z]$ be a polynomial with coefficients $a_n\in R(\Gal(E/\Q))$ in the ring of virtual characters, that is the ring of integer linear combinations of irreducible characters of $E/\Q$. Then the Euler product
    \[
    \prod_p h_p(p^{-s})^{-1} = \prod_p \left(1 + \sum a_n(\Fr_p(E/\Q)) p^{-ns}\right)^{-1}
    \]
    has a meromorphic continuation to the right half-plane ${\rm Re}(s)>0$.

    If the roots of $h_p(z)$ is unitary then this Euler product can be meromorphically continued to $\C$, otherwise ${\rm Re}(s) = 0$ is the natural boundary of this function.
\end{theorem}

Just like Estermann, Kurokawa and Moroz prove this result by applying the Factorization Method to relate the Euler product to a zeta product
\[
    \prod_n \prod_\rho L(ns,\rho)^{b_{n,\rho}}
\]
for some sequence $(b_{n,\rho})$, up to the ramified places of the Galois representations $\rho$. This product is finite if and only if $h_p(z)$ is of the form $\det(I - \psi(\Fr_p)z)$ for some finite dimensional Galois representation $\psi$ and almost all places $p$, which is how Kurakawa and Moroz define \textbf{unitary}.

We also give a self-contained proof of the meromorphic continuations in Theorem \ref{thm:km}. The existence of an analytic continuation to a branch cut of ${\rm Re}(s) > 0$ is proven in Theorem \ref{thm:existence_Frob_coef}, and it is a consequence of Theorem \ref{thm:explicit_singularities_Frob_coef} that all of the singularities are poles so that no branch cut is necessary. We again summarize the main idea behind the proof that ${\rm Re}(s)=0$ is the natural boundary in Section \ref{sec:natural_boundary}, though we leave many of the details to the original papers.

The author was unable to find a result for general Frobenian Euler products analogous to Dahlquist's Theorem \ref{thm:dahlquist} in the literature, although a proof is readily constructed by combining key features of Dahlquist's and Moroz's work. The self-contained proof of Theorem \ref{thm:existence_Frob_coef} we present readily implies the following result:

\begin{theorem}\label{thm:frob_dahlquist}
    Let $p\mapsto Q_p(z)$ be a Frobenian map over a finite extension $E/\Q$, where each function $Q_p(z)$ is holomorphic on (a branch cut of) the open unit disk, save for a set of isolated singularities, and is holomorphic at $z=0$ with $Q(0)=1$. Then the Euler product
    \[
    \mathcal{Q}(s) = \prod_p Q_p(p^{-s})
    \]
    has an analytic continuation to (a branch cut of) the right halfplane ${\rm Re}(s) > 0$, save for a set of isolated singularities.
\end{theorem}

Just like with Dahlquist's results, we remark that it is necessary to allow for singularities of non-integral order in Theorem \ref{thm:frob_dahlquist}. In particular, it is important to avoid using the word ``meromorphic" for this result.

\begin{remark}
    It is reasonable to to expect that, unless $Q_p(z)$ is a unitary polynomial, ${\rm Re}(s) = 0$ is the natural boundary for $\mathcal{Q}(s)$ in Theorem \ref{thm:frob_dahlquist} (up to branch cuts). The main ideas described in Section \ref{sec:natural_boundary} apply equally well to this case, but the details of the argument are outside the scope of this paper.
\end{remark}

\section{Constructing Analytic Continuations}

The analytic continuation to (a branch cut of) the right-half plane ${\rm Re}(s) > 0$ is done up to ${\rm Re}(s) > \epsilon$ for arbitrarily small $\epsilon$. This is so that $\zeta(ns)$ is absolutely convergent on this region for all but finitely many values of $n\in \Z^+$, or in other words only finitely many of these functions can contribute a pole.

\begin{theorem}\label{thm:existence_const_coef}
    Let $Q(z)$ be holomorphic on (a branch cut of) the open unit disk, save for a set of isolated singularities, which is holomorphic at $z=0$ with $Q(0)=1$. Let $r$ be the smallest magnitude of a singularity of $Q(z)$ and let $\mathcal{Q}(s) = \prod Q(p^{-s})$ be the corresponding Euler product.
    \begin{enumerate}[(i)]
        \item There exists a sequence of complex numbers $b_n(Q) = b_n$ such that $\log Q(z) = -\sum b_n\log(1-z^n)$ as formal power series centered at $z=0$, and
        \item For each $\epsilon > 0$,
        \[
            \mathcal{Q}(s) \prod_{p\le r^{-1/\epsilon}}Q(p^{-s})^{-1}\prod_{n\le 1/\epsilon}\zeta(ns)^{-b_n}
        \]
        is an absolutely convergent Dirichlet series (and so holomorphic) on the right-half plane ${\rm Re}(s) > \epsilon$.
    \end{enumerate}
\end{theorem}

The existence of an analytic continuation to ${\rm Re}(s)>0$ (up to a branch cut) in Estermann's Theorem \ref{thm:estermann} and Dahlquist's Theorem \ref{thm:dahlquist} follows immediately from this result by taking $\epsilon \to 0$. Moreover, we see that a branch cut is needed if and only if one of the $Q(p^{-s})$ requires a branch cut or one of the $b_n$ is not an integer. While Estermann and Dahlquist did describe the sequence $b_n$ as we do in Theorem \ref{thm:existence_const_coef}(i), this sequence is evident from their proofs.

We are now able to see that the Factorization Method is, secretly, linear algebra. The steps of the method are calculating the first nonzero coefficient in the expansion of $\log Q(z)$ with respect to the basis $\{-\log(1-z^n)\}_{n\ge 1}$. Step (1) and (2) are identifying the smallest basis vector that has a nonzero coefficient, step (3) is factoring the Euler product as
\[
    \mathcal{Q}(s) = \zeta(as)^{b_a} \left(\mathcal{Q}(s)\zeta(as)^{-b_a}\right),
\]
and step (4) is proving a special case of Theorem \ref{thm:existence_const_coef}, that $\mathcal{Q}(s)\zeta(as)^{-b_a}$ converges absolutely on a larger region. Iterating the Factorization Method finds each successive coefficient. By doing the basis decomposition all at once at the start, Theorem \ref{thm:existence_const_coef} gives the analytic continuation that results from applying the Factorization Method ``infinitely many times".

We will discuss some of the intricacies of the infinite dimensional linear algebra we need in Section \ref{sec:inf_lin_alg}, but modulo these details we are ready to prove this result.

\begin{proof}
    $Q(z)$ is equal to a power series in some neighborhood of $z=0$ with constant term $1$, so that we can view $\log Q(z)\in z\C[[z]]$ as a formal power series. Part (i) follows from Proposition \ref{prop:basis_const}, which shows that $\{-\log(1-z^n)\}_{n\ge 1}$ is a Schauder basis of $z\C[[z]]$ (i.e. a basis allowing for infinite linear combinations). This shows the existence of the sequence $b_n(Q)=b_n$.

    Fix $\epsilon > 0$. We then expand
    \begin{align*}
        \mathcal{Q}(s) \prod_{p\le r^{-1/\epsilon}}Q(p^{-s})^{-1}\prod_{n\le 1/\epsilon}\zeta(ns)^{-b_n} &= \prod_{p\le r^{-1/\epsilon}}\left(\prod_{n\le 1/\epsilon}(1-p^{-ns})^{b_n}\right) \prod_{p> r^{-1/\epsilon}}\left(Q(p^{-s})\prod_{n\le 1/\epsilon}(1-p^{-ns})^{b_n}\right).
    \end{align*}
    The product over $p\le r^{-1/\epsilon}$ is finite, and so necessarily absolutely convergent. For ${\rm Re}(s) > \epsilon$ and $p>r^{-1/\epsilon}$, we have $|p^{-s}| < r$. $Q(z)$ is equal to its Taylor series on the ball $|z|<r$ by $Q(z)$ not having any singularities in this region, and thus so is $\log Q(z)$. We can then write
    \begin{align*}
        Q(p^{-s})\prod_{n\le 1/\epsilon}(1-p^{-ns})^{b_n} &= \exp\left(\log Q(p^{-s}) + \sum_{n\le 1/\epsilon} b_n\log(1-(p^{-s})^n)\right)\\
        &=\exp\left(-\sum_{n> 1/\epsilon} b_n\log(1-(p^{-s})^n)\right).
    \end{align*}
    The series inside the exponential converges absolutely for $|p^{-s}| < r$ because it is a sum of two series that converge absolutely on this region. By the remainder theorem from calculus, it follows that
    \[
        -\sum_{n> 1/\epsilon} b_n\log(1-z^n) = O(z^{1/\epsilon})
    \]
    on the region $|z|<r$, which implies
    \begin{align*}
        Q(p^{-s})\prod_{n\le 1/\epsilon}(1-p^{-ns})^{b_n} &= \exp\left(O(p^{-s/\epsilon})\right) = 1 + O(p^{-s/\epsilon})
    \end{align*}
    on the region $|p^{-s}| < r$. As discussed above, this includes the region ${\rm Re}(s) > \epsilon$ for each $p > r^{-1/\epsilon}$. Thus
    \begin{align*}
        \prod_{p>r^{-1/\epsilon}}\left(Q(p^{-s})\prod_{n\le 1/\epsilon}(1-p^{-ns})^{b_n}\right) &= \prod_{p>r^{-1/\epsilon}}\left(1 + O(p^{-s/\epsilon})\right),
    \end{align*}
    which converges absolutely as ${\rm Re}(s) > \epsilon$ implies ${\rm Re}(s)/\epsilon > 1$. This concludes the proof.
\end{proof}

We give an analogous result for Frobenian Euler products.

\begin{theorem}\label{thm:existence_Frob_coef}
    Let $p\mapsto Q_p(z)$ be Frobeian with respect to the finite extension $E/\Q$, where each $Q_p(z)$ is holomorphic on (a branch cut of) the open unit disk, save for a set of isolated singularities, which is holomorphic at $z=0$ with $Q_p(0)=1$. Let $r$ be the smallest magnitude of a singularity of $Q_p(z)$ for any $p$ and let $\mathcal{Q}(s) = \prod Q_p(p^{-s})$ be the corresponding Euler product.
    \begin{enumerate}[(i)]
        \item There exists a sequence of complex numbers $b_{n,\rho}(Q) = b_{n,\rho}$ indexed by positive integers $n$ and irreducible characters $\rho$ of $\Gal(E/\Q)$ such that for all but finitely many primes
        \[
            \log Q_p(z) = -\sum_{n,\rho} b_{n,\rho}\log\det(I-\rho(\Fr_p)z^n)
        \]
        as formal power series centered at $z=0$, and
        \item For each $\epsilon > 0$,
        \[
            \mathcal{Q}(s) \prod_{p\le r^{-1/\epsilon}}Q_p(p^{-s})^{-1}\prod_{n\le 1/\epsilon}\prod_\rho L(ns,\rho)^{-b_{n,\rho}}
        \]
        is an absolutely convergent Dirichlet series (and so holomorphic) on the right-half plane ${\rm Re}(s) > \epsilon$.
    \end{enumerate}
\end{theorem}

Theorem \ref{thm:existence_Frob_coef} implies the existence of an analytic continuation to (a branch cut of) ${\rm Re}(s) > 0$ in the statements of Kurokawa's and Moroz's Theorem \ref{thm:km} as well as in the original Theorem \ref{thm:frob_dahlquist} generalizing Dahlquist's work. Just like in the constant coefficients case, we see that a branch cut is needed if and only if one of the $Q_p(p^{-s})$ requires a branch cut or one of the $b_{n,\rho}$ is not an integer.

The proof of Theorem \ref{thm:existence_Frob_coef} is completely analogous to that of Theorem \ref{thm:existence_const_coef}. We include it for the sake of completeness.

\begin{proof}
    For each prime $Q_p(z)$ is equal to a power series in some neighborhood of $z=0$ with constant term $1$, so that for all but finitely many primes we can view $\log Q_p(z)\in zR(\Gal(E/\Q))[[z]]$ as a formal power series, where $R(\Gal(E/\Q))$ is the ring of virtual characters $\Gal(E/\Q)\to \C$. Part (i) follows from Proposition \ref{prop:basis_Frob}, which shows that $\{-\log\det(I-\rho(\Fr_p)z^n)\}_{n\ge 1}$ is a Schauder basis of $zR(\Gal(E/\Q))[[z]]$ (i.e. a basis allowing for infinite linear combinations). This shows the existence of the sequence $b_{n,\rho}(Q)=b_{n,\rho}$.

    Fix $\epsilon > 0$. We then expand
    \begin{align*}
        &\mathcal{Q}(s) \prod_{p\le r^{-1/\epsilon}}Q_p(p^{-s})^{-1}\prod_{n\le 1/\epsilon}\prod_\rho L(ns,\rho)^{-b_n}\\
        &=\prod_{p\le r^{-1/\epsilon}}\left(\prod_{n\le 1/\epsilon}\prod_\rho\det(I-\rho(\Fr_p|V^{I_p})p^{-ns})^{b_{n,\rho}}\right) \prod_{p> r^{-1/\epsilon}}\left(Q(p^{-s})\prod_{n\le 1/\epsilon}\prod_\rho\det(I-\rho(\Fr_p|V^{I_p})p^{-ns})^{b_n}\right).
    \end{align*}
    The product over $p\le r^{-1/\epsilon}$ is finite, and so necessarily absolutely convergent. For ${\rm Re}(s) > \epsilon$ and $p>r^{-1/\epsilon}$, we have $|p^{-s}| < r$. $Q_p(z)$ is equal to its Taylor series on the ball $|z|<r$ by $Q_p(z)$ not having any singularities in this region, and thus so is $\log Q_p(z)$. We can then write
    \begin{align*}
        Q_p(p^{-s})\prod_{n\le 1/\epsilon}\prod_\rho\det(I-\rho(\Fr_p|V^{I_p})p^{-ns})^{b_n} &= \exp\left(\log Q_p(p^{-s}) + \sum_{n\le 1/\epsilon}\sum_{\rho} b_{n,\rho}\log\det(I-\rho(\Fr_p|V^{I_p})(p^{-s})^n)\right)\\
        &=\exp\left(-\sum_{n> 1/\epsilon}\sum_{\rho} b_{n,\rho}\log\det(I-\rho(\Fr_p|V^{I_p})(p^{-s})^n\right).
    \end{align*}
    The series inside the exponential converges absolutely for $|p^{-s}| < r$ because it is a sum of two series that converge absolutely on this region. By the remainder theorem from calculus, it follows that
    \[
        -\sum_{n> 1/\epsilon} \sum_{\rho} b_{n,\rho}\log\det(I-\rho(\Fr_p|V^{I_p})z^n= O(z^{1/\epsilon})
    \]
    on the region $|z|<r$, which implies
    \begin{align*}
        Q_p(p^{-s})\prod_{n\le 1/\epsilon}\prod_\rho\det(I-\rho(\Fr_p|V^{I_p})p^{-ns})^{b_n} &= \exp\left(O(p^{-s/\epsilon})\right) = 1 + O(p^{-s/\epsilon})
    \end{align*}
    on the region $|p^{-s}| < r$. As discussed above, this includes the region ${\rm Re}(s) > \epsilon$ for each $p > r^{-1/\epsilon}$. Thus
    \begin{align*}
        \prod_{p> r^{-1/\epsilon}}\left(Q(p^{-s})\prod_{n\le 1/\epsilon}\prod_\rho\det(I-\rho(\Fr_p|V^{I_p})p^{-ns})^{b_n}\right) &= \prod_{p>r^{-1/\epsilon}}\left(1 + O(p^{-s/\epsilon})\right),
    \end{align*}
    which converges absolutely as ${\rm Re}(s) > \epsilon$ implies ${\rm Re}(s)/\epsilon > 1$. This concludes the proof.
\end{proof}

\section{Some Infinite Dimensional Linear Algebra}\label{sec:inf_lin_alg}

The Factorization Method is, secretly, a basis decomposition. In the constant coefficient case, the expression $\log Q(z) = -\sum b_n \log(1-z^n)$ is an expression of $\log Q(z)$ as an element of the span of $\{-\log(1-z^n)\}$. However, it is important to recognize that this is the span when allowing convergent infinite linear combinations.

Let $V$ be a normed vector space. A \textbf{Schauder basis} for $V$ is a set of vectors $\{v_n\}$ such that every vector in $V$ can be written as a convergent linear combination of the $\{v_n\}$, and the only linear combination of $\{v_n\}$ that converges to $0$ is the trivial one.

Given a power series $f(z)$, define ${\rm ord}(f)$ to be the smallest degree term with a nonzero coefficient. This makes $z\C[[z]]$ into a normed vector space with norm given by
\[
    ||f(z)|| = \frac{1}{{\rm ord}(f)}.
\]
Consider the sequence of monomials $\{z^n\}_{n\ge 1}$ in $z\C[[z]]$. By definition of the ring of power series, these vectors form a Schauder basis for $z\C[[z]]$.

\begin{proposition}\label{prop:basis_const}
    The sequence $\{-\log(1-z^n)\}_{n\ge 1}$ is a Schauder basis for $z\C[[z]]$.
\end{proposition}

Thus, any power series $\log Q(z) \in z\C[[z]]$ can be written as a convergent linear combination of $-\log(1-z^n)$. The coefficients of this linear combination are precisely the sequence $b_n$. We can prove something similar for Frobenian power series.

\begin{proposition}\label{prop:basis_Frob}
    Let $E/\Q$ be a finite extension and $R(\Gal(E/\Q))$ the ring of complex valued virtual characters. The sequence $\{-\log\det(I - \rho(\Fr_p|V^{I_p})z^n)\}_{n\ge 1,\rho}$ is a Schauder basis for $zR(\Gal(E/\Q))[[z]]$, where $\rho:\Gal(E/\Q)\to {\rm GL}(V)$ varies over the irreducible representations of $\Gal(E/\Q)$.
\end{proposition}

These propositions are proven by showing that the linear transformation from the standard basis to these vectors is lower triangular in a certain infinite dimensional sense with nonzero diagonal. Just like their finite dimensional counterparts, these infinite dimensional lower triangular matrices are invertible.

\begin{lemma}\label{lem:lowertri}
Let $W$ be a vector space with a Schauder basis $\{w_i\}_{i\in I}$ indexed by a partially ordered set $I$. Suppose further that $I$ satisfies a Northcott property, namely that for each $i\in I$ the set $\{j\in I : j < i\}$ is finite. We call a matrix $T = [t_{i,j}]$ \textbf{lower triangular} if $t_{i,j}\ne 0$ implies $i \ge j$.

Every lower triangular matrix whose diagonal entries are nonzero is invertible.
\end{lemma}

\begin{proof}
Consider the equation $T\vec{v} = \vec{c}$. We prove that this equation has a unique solution by strong induction on the index set $I$. We note that the Northcott property implies a version of being well-ordered: If $S\subset I$ is nonempty, then there exists some $i\in S$. The set
\[
S\cap \{j \le i : j\in I\}
\]
is finite and is closed under $<$ as a subset of $S$, and therefore must contain a minimal element of $S$.

We construct the coordinates of the vector $\vec{v}$ by induction. For the base case, let $i_0\in I$ be any minimal element (note that this may not be unique). The $i_0$th row of $T$ satisfies
\[
t_{i_0,j} = \begin{cases}
\text{nonzero} & j=i_0\\
0 & j\ne i_0,
\end{cases}
\]
because $i_0$ is the only index $\le i_0$. Therefore, the product along the $i_0$th row is given by
\[
c_{i_0} = \sum_{j} t_{i_0,j}v_{j} = t_{i_0,i_0} v_{i_0}.
\]
Thus $v_{i_0} = t_{i_0,i_0}^{-1} c_{i_0}$ exists and is unique. This is true regardless of which minimal element we choose, so the base case may be considered to be ``all minimal elements of $I$".

We proceed via strong induction. Choose some $i\in I$. Suppose that for each $j < i$, there is a unique solution for $v_j$. Then multiplying through the $i$th row gives
\[
c_i = \sum_j t_{i,j}v_j = \sum_{j\le i}t_{i,j} v_j
\]
by $T$ lower triangular. We note that, by the definition of the ordering on $I$, this sum is finite so we do not need to worry about convergence issues. Solving for $v_i$ we find that
\[
v_i = t_{i,i}^{-1}\left(c_i - \sum_{j<i}t_{i,j} v_j\right),
\]
which exits and is unique by the inductive hypothesis.

The result then follows from strong induction on a well-ordered set and gives an iterative formula for each $v_i$ value.
\end{proof}

We can now prove the propositions.

\begin{proof}[Proof of Proposition \ref{prop:basis_const}]
    Let $T:z\C[[z]]\to z\C[[z]]$ be the linear transformation sending $Tz^n = -\log(1-z^n)$. If we write $T=[t_{i,j}]$ as a matrix with respect to the standard basis $\{z^n\}_{n\ge 1}$, then
    \begin{align*}
        \sum_{j=1}^{\infty} t_{j,n}z^n = Tz^n = -\log(1-z^n) = \sum_{k=1}^{\infty} \frac{1}{k}z^{kn}.
    \end{align*}
    Thus, $t_{j,n} = 0$ if $j\ne kn$ for some integer $k\ge 1$, so in particular if $j<n$. This implies $T$ is lower triangular. Moreover, $t_{n,n} = 1$ is nonzero for each $n\ge 1$, which by Lemma \ref{lem:lowertri} implies that $T$ is invertible. Thus, $\{-\log(1-z^n)\}_{n\ge 1}$ is the image of a basis under an invertible linear transformation, and so too must be a basis.
\end{proof}

\begin{proof}[Proof of Proposition \ref{prop:basis_Frob}]
    The vector space $zR(\Gal(E/\Q))[[z]]$ has standard basis given by $\{{\rm tr}\rho(\Fr_p|V^{I_p}) z^n\}_{n\ge 1,\rho}$, where $\rho:\Gal(E/\Q)\to V$ varies over the irreducible representations of $\Gal(E/\Q)$. This is a basis because the irreducible characters form a basis for $R(\Gal(E/\Q))$.
    
    Let $T:zR(\Gal(E/\Q))[[z]]\to zR(\Gal(E/\Q))[[z]]$ be the linear transformation sending $T{\rm tr}\rho(\Fr_p|V^{I_p})z^n = -\log\det(I-\rho(\Fr_p|V^{I_p})z^n)$. If we write $T=[t_{(i,\gamma),(j,\psi)}]$ as a matrix with respect to the standard basis $\{{\rm tr}\rho(\Fr_p|V^{I_p})z^n\}_{n\ge 1,\rho}$, then
    \begin{align*}
        \sum_{j=1}^{\infty} t_{(i,\gamma),(n,\rho)}{\rm tr}\rho(\Fr_p|V^{I_p})z^n = Tz^n = -\log\det(I-\rho(\Fr_p|V^{I_p})z^n) = 1 + {\rm tr}\rho(\Fr_p|V^{I_p})z^n + O(z^{2n}).
    \end{align*}
    Put a partial ordering on the set
    \[
        \{(i,\gamma):i\in \Z^+,\gamma\text{ irreducible representation of }\Gal(E/\Q)\}
    \]
    given by $(i,\gamma) < (j,\psi)$ if $i < j$. This set is well-ordered by $E/\Q$ finite, and so admitting at most finitely many irreducible representations. Thus, $t_{(i,\gamma),(n,\rho)} = 0$ if $i\ne kn$ for some integer $k\ge 1$, so in particular if $i<n$. By definition, this means that $t_{(i,\gamma),(n,\rho)} = 0$ if $(i,\gamma) < (n,\rho)$ which implies $T$ is lower triangular. Moreover, $t_{(n,\rho),(n,\rho)} = 1$ is nonzero for each $n\ge 1$, which by Lemma \ref{lem:lowertri} implies that $T$ is invertible. Thus, $\{-\log\det(I-\rho(\Fr_p|V^{I_p})z^n)\}_{n\ge 1,\rho}$ is the image of a basis under an invertible linear transformation, and so too must be a basis.
\end{proof}

\section{Determining the Natural Boundary}\label{sec:natural_boundary}

The primary goal of Estermann's work, and the work that followed, was the determination of the natural boundary. This is done by proving that every point on the line ${\rm Re}(s)=0$ is a limit of zeros or singularities, and therefore are essential singularities.

Study of the natural boundary of Dirichlet series has continued far beyond the cases discussed in this paper. One case of particular interest is that of Euler products of the form
\[
    \mathcal{W}(s) := \prod_p W(p,p^{-s}),
\]
where $W(X,Y)\in \C(X,Y)$ is a rational function. Perhaps surprisingly, the natural boundary of $\mathcal{W}(s)$ is still not known for every such rational function $W(X,Y)$. Similar to Example \ref{ex:growing_coef}, $\mathcal{W}(s)$ can be factored as an infinite product of functions of the form $\zeta(ns+m)$ for integers $n,m\in \Z$. The horizontal shift makes it more difficult to exclude the possibility that the zeros of each factor are cancelled out by zeros of other factors. The zeta functions of certain nilpotent groups are naturally given by this type of Euler product, and for this reason the analytic theory of the natural boundary for this type of Euler product is summarized in the book ``Zeta Functions of Groups and Rings'' by du Sautoy and Woodward \cite[Section 5]{duSautoy-Woodward2008}.

The conjectural location of the natural boundary for $\mathcal{W}(s)$ when $W(X,Y)$ is a polynomial can be found in \cite[Conjecture 1.4]{bhowmik2010}. Some references concerning the natural boundaries for these and other Dirichlet series include \cite{bhowmik-schlage-puchta2007,essouabri_velásquez_castañón_2020,han_ki_park_2021}. Of particular interest is \cite{bhowmik-schlage-puchta2016}, where it is shown that the natural boundary need not be of the form ${\rm Re}(s) = \beta$ for a real number $\beta$, and \cite{bhowmik-schlage-puchta2010}, where it is shown that there must exist Euler products $\mathcal{W}(s)$ of this form for which Estermann's argument cannot work to produce the natural boundary.

The analog of the natural boundary for Dirichlet series of multiple variables has also been considered, in particular for those arising as the height zeta functions of varieties. For some examples, see \cite{delabarre_2013,delabarre2014,essouabri2012height,bhowmik_essouabri_lichtin_2007}.

The full details of arguments proving the existence of natural boundaries are outside the scope of this paper, as we are primarily concerned with the explicit construction of a continuation. However, we can briefly summarize the main idea in the language of Theorem \ref{thm:existence_const_coef} and Theorem \ref{thm:existence_Frob_coef}. We refer the reader to the corresponding papers \cite{estermann_1928,dahlquist_1952,kurokawa_1978,kurokawa_1986I,kurokawa_1986II,moroz_1988} for the details.

\begin{proof}[Sketch of the Argument Determining the Natural Boundary]
    Let $Q(z)$ be as in Theorem \ref{thm:existence_const_coef} (respectively Theorem \ref{thm:existence_Frob_coef}), with corresponding sequence $b_n$ (resp. $b_{n,\rho}$). Assume further that $Q(z)$ is holomorphic on a branch cut of the entire complex plane, rather than just the open unit disk.

    If the sequence $b_n$ (resp. $b_{n,\rho}$) has at most finitely many nonzero terms, then it is clear from the proofs of Theorem \ref{thm:existence_const_coef} (resp. Theorem \ref{thm:existence_Frob_coef}) that $\mathcal{Q}(s)$ factors as a finite product of complex powers of Artin $L$-functions times a function which is holomorphic on a branch cut of $\C$. Thus, $\mathcal{Q}(s)$ itself is holomorphic on a branch cut of the the entire complex plane.

    Otherwise, for each $n$ such that $b_n$ (resp. $b_{n,\rho}$) is nonzero the factor $\zeta(ns)^{b_n}$ (resp. $L(ns,\rho)^{b_n}$) contributes a zero or an isolated singularity at each root of $\zeta(ns)$ (resp. $L(ns,\rho)$). For each $n$, consider the set
    \[
        D_n(\delta,t_0) = \left\{s=\sigma+it : \frac{1}{2n+1}\le \sigma \le \frac{1}{2n-1},\ t_0\le t\le t_0+\delta\right\}
    \]
    for each $t_0$ and $\delta$. It suffices to show that for all $\delta,t_0$, there exist infinitely many $n$ for which $D_n(\delta,t_0)$ contains a zero or singularity of $\mathcal{Q}(s)$. This would imply that $s = 0 + it_0$ is a limit point of zeros or singularities of $\mathcal{Q}(s)$, and therefore must be an essential singularity of $\mathcal{Q}(s)$. Doing this for each $t_0$ implies that ${\rm Re}(s) = 0$ is a wall of essentialy singularities, past which we cannot continue $\mathcal{Q}(s)$.

    It is known that 100\% of the zeros of $\zeta(s)$ (resp. $L(s,\rho)$) in the critical strip have real part in $[1/2-\epsilon,1/2+\epsilon]$. The set $D_n(\delta,t_0)$ intersects the critical line of $\zeta(ms)$ (resp. $L(ms,\rho)$) for at most one value of $m$, so that $D_n(\delta,t_0)$ is expected to contain asymptotically more zeros of $\zeta(ns)$ (resp. $L(ns,\rho)$) than for $\zeta(ms)$ (resp. $L(ms,\rho)$) with $m\ne n$. Thus, no matter what the sequence $\{b_n\}$ is, as long as $b_n\ne 0$ we expect at least one root of $\zeta(ns)$ (resp. $L(ns,\rho)$) to not be cancelled out in the product of $L$-functions giving $\mathcal{Q}(s)$. This corresponds to a zero or singularity of $\mathcal{Q}(s)$ in $D_n(\delta,t_0)$.

    There is more work needed to make this rigorous, in particular checking that the zeros of $L(s,\rho)$ and $L(s,\psi)$ do not have significant overlap for different representations $\rho$ and $\psi$. The bunching up of zeros near the critical strip of $\zeta(ns)$ as $n\to \infty$ is the main idea behind these proofs, but the actual proofs involve deep results about the locations of zeros of Artin $L$-functions.
\end{proof}

\part{Explicit Descriptions of Singularities}\label{part:explicit}

\section{Statement of Explicit Results}

The singularities in Theorem \ref{thm:existence_const_coef} are determined by the singularities of $Q(p^{-s})$ or of $\zeta(ns)^{b_n(Q)}$. The singularities of $Q(p^{-s})$ (if there are any) are determined by the choice of function $Q(z)$. Meanwhile, the singularities of $\zeta(ns)^{b_n(Q)}$ are determined by the pole or zeros of $\zeta(s)$ together with the value of $b_n(Q)$.

Thus, to understand the singularities of $\mathcal{Q}(s)$, it suffices to understand
\begin{itemize}
    \item The zeros and singularities of $Q(z)$,
    \item The zeros and singularities of $\zeta(s)$, and
    \item the values of the sequence $b_n(Q)$.
\end{itemize}
In the cases of interest, the zeros and singularities of $Q(z)$ are often evident from the function being considered. While we do not completely understand the zeros of $\zeta(s)$, in many applications it is sufficient to know that the nontrivial zeros all lie in the critical strip. Problems that require more information about the zeros of $\zeta(s)$ may use the zero-free region or prove results conditional on the Riemann Hypothesis.

That leaves the sequence $b_n(Q)$. Theorem \ref{thm:existence_const_coef}(i) gave this sequence as the coefficients of a basis decomposition. By taking more care to be explicit with the linear algebra, we can prove some equivalent descriptions for this sequence.

\begin{theorem}\label{thm:explicit_singularities_const_coef}
    Let $Q(z)\in 1+z\C[[z]]$ be a power series. Let $b_n = b_n(Q)$ be the unique sequence of complex numbers such that
    \[
        \log Q(z) = -\sum_{n=1}^{\infty} b_n \log(1-z^n).
    \]
    Then the following hold:
    \begin{enumerate}[(i)]
        \item Write $\log Q(z)=\sum q_n z^n$. Then $b_1=q_1$, and for $n>1$ the sequence is given by the recursive formula
        \[
            b_n = q_n - \sum_{\substack{d\mid n\\d>1}}\frac{1}{d}b_{n/d}.
        \]
        \item If $Q(z) \in 1 + z\Z[[z]]$, then $b_n\in \Z$.
    \end{enumerate}
\end{theorem}

Theorem \ref{thm:explicit_singularities_const_coef}(i) is useful for performing practical calculations. In applications, it is often only necessary to know finitely many terms of the sequence $b_n$, which provides enough information to determine the locations and orders of singularities in some right-half plane ${\rm Re}(s) > \delta > 0$. Theorem \ref{thm:explicit_singularities_const_coef}(ii) is useful for a more abstract reason - we know by Theorem \ref{thm:existence_const_coef} that if $Q(z)$ is meromorphic and $b_n\in \Z$ for each $n$, then $\mathcal{Q}(s)$ is meromorphic on ${\rm Re}(s)>0$. No branch cuts needed. For example, Theorem \ref{thm:explicit_singularities_const_coef}(ii) implies that the continuation in Estermann's Theorem \ref{thm:estermann} is a meromorphic continuation.

\begin{remark}
    Theorem \ref{thm:explicit_singularities_const_coef}(ii) comes from an explicit formula giving the values of $b_n$ in terms of the coefficients of $\log Q(z)$, as opposed to the recursive formula in Theorem \ref{thm:explicit_singularities_const_coef}(i). It is possible to extract a closed form expression for $b_n$ from the proof of Theorem \ref{thm:explicit_singularities_const_coef}, but we opt not to include this as the expression is more complicated than it is worth.
\end{remark}

We now use Theorem \ref{thm:explicit_singularities_const_coef} to quickly give the locations of several singularities in Example \ref{ex:intro}.

\begin{BOXexample}\label{ex:intro_more_poles}
    Consider the Euler product in Example \ref{ex:intro},
    \[
    \mathcal{Q}(s) = \prod_{p}(1 + 2p^{-s}).
    \]
    This corresponds to the entire function $Q(z) = 1 + 2z$, which is a polynomial with integer coefficients. Thus Theorem \ref{thm:existence_const_coef} implies that $\mathcal{Q}(s)$ has a meromorphic continuation to the right halfplane ${\rm Re}(s) > 0$.

    Let's explicitly give the meromorphic continuation of $\mathcal{Q}(s)$ to the region ${\rm Re}(s) > 1/6$. Noting that the radius of convergence for $\log Q(z)=\log(1+2z)$ is $r=1/2$, Theorem \ref{thm:existence_const_coef} implies that
    \[
    G(s) = \mathcal{Q}(s) \prod_{p\le 2^6} (1+2p^{-s})^{-1}\prod_{n=1}^6 \zeta(ns)^{-b_n}
    \]
    is absolutely convergent on this region. Thus
    \[
    \mathcal{Q}(s) = G(s)\prod_{p\le 2^6} (1+2p^{-s})\prod_{n=1}^6 \zeta(ns)^{b_n}
    \]
    in the region ${\rm Re}(s) > 1/6$. Given that $Q(z)\ne 0$ on the positive real line, it follows that $G(s)\ne 0$ and $Q(p^{-s})=1+2p^{-s}\ne 0$ for real $s\in (1/6,\infty]$. It now suffices to calculate $b_1,b_2,...,b_6$ to compute the poles.

    We compute the Taylor series
    \[
    \log Q(z) = \sum_{n=1}^\infty \frac{(-1)^{n-1}2^n}{n}z^n = 2z - 2z^2 + \frac{8}{3}z^3 - 4z^4 + \frac{32}{5}z^5 - \frac{32}{3}z^6 + \cdots,
    \]
    so that the recursive formula in Theorem \ref{thm:explicit_singularities_const_coef}(i) implies
    \begin{align*}
        b_1 &= 2\\
        b_2 &= -2 - \frac{1}{2}b_1 = -3\\
        b_3 &= \frac{8}{3} - \frac{1}{3}b_1 = 2\\
        b_4 &= -4 - \frac{1}{4}b_1 - \frac{1}{2}b_2 = -3\\
        b_5 &= \frac{32}{5} - \frac{1}{5}b_1 = 6\\
        b_6 &= \frac{-32}{3} - \frac{1}{6}b_1 - \frac{1}{3}b_2 - \frac{1}{2}b_3 = -11.
    \end{align*}
    Thus, we have proven that
    \[
    \mathcal{Q}(s) = \zeta(s)^2 \zeta(2s)^{-3}\zeta(3s)^{2}\zeta(4s)^{-3}\zeta(5s)^6\zeta(6s)^{-11} G(s)\prod_{p\le 2^6} (1+2p^{-s}),
    \]
    where $G(s)$ is holomorphic and nonzero on the region ${\rm Re}(s) > 1/6$. We now see that all singularities in this region are inherited from the zeta functions, although some may be (partially) cancelled out by the (non-real) zeros of $1+2p^{-s}$ for $p\le 2^6$.
\end{BOXexample}

A similar result holds for Frobenian Euler products, which implies that the continuation in Kurakawa's and Moroz's Theorem \ref{thm:km} is a meromorphic continuation.

\begin{theorem}\label{thm:explicit_singularities_Frob_coef}
    Let $Q_p(z)\in 1+zR(\Gal(E/\Q))[[z]]$ be a power series, where $E/\Q$ is a finite extension and $R(\Gal(E/\Q))$ is the ring of virtual characters $\Gal(E/\Q)\to \C$. Let $b_{n,\rho} = b_{n,\rho}(Q)$ be the unique sequence of complex numbers indexed by positive integers $n$ and irreducible representations $\rho:\Gal(E/\Q)\to {\rm GL}(V)$ such that
    \[
        \log Q_p(z) = -\sum_{n=1}^{\infty}\sum_{\rho} b_{n,\rho} \log\det(I-\rho(\Fr_p|V^{I_p})z^n).
    \]
    Then the following hold:
    \begin{enumerate}[(i)]
        \item Write $\log Q_p(z)=\sum q_{n,\rho}{\rm tr}\rho(\Fr_p|V^{I_p})z^n$ and define the recursive sequence $r_{n,\rho}$ by $r_{1,\rho} = q_{1,\rho}$ for each $\rho$ and
        \[
            r_{n,\rho} = q_{n,\rho} - \sum_{\substack{d\mid n\\d>1}}\sum_{\gamma\text{ irrep.}}\frac{\langle\rho^{\otimes d},\gamma\rangle}{d}r_{n/d,\gamma}
        \]
        for each $\rho$. Then the sequence $b_{n,\rho}$ is given by
        \[
            b_{n,\rho} = \sum_{d\mid n}\sum_{\gamma\text{ irrep.}} \langle \gamma^{\otimes_{\rm cyc} d},\rho\rangle r_{n/d,\gamma},
        \]
        where the cyclic tensor power $\gamma^{\otimes_{\rm cyc} d}$ is defined in Definition \ref{def:cyclic_tensor}.
        \item If $Q_p(z) \in 1 + zR(\Gal(E/\Q),\Z)[[z]]$ where $R(\Gal(E/\Q),\Z)\subseteq R(\Gal(E/\Q))$ is the set of integer linear combinations of the irreducible characters, then $b_n\in \Z$.
    \end{enumerate}
\end{theorem}

In order to prove Theorem \ref{thm:explicit_singularities_const_coef} and Theorem \ref{thm:explicit_singularities_Frob_coef}, we will need new results in linear algebra, combinatorics, and representation theory.

\section{The Log Factorization Theorem}

Let $\N^\N$ be the set of sequence of natural numbers. For any $\mathbf{n}\in \N^\N$, define
\[
    |\mathbf{n}| := \sum_{i=0}^{\infty} n_i.
\]
We let $\N^{\N}_0$ be the subset of $\N^\N$ for which $|\mathbf{n}|<\infty$.

The following theorem formally takes place in the power series ring $\C[[x_0,x_1,x_2,...]]$ in countably many variables. The monomials in this space are parametrized by $\N^\N_0$ via the correspondence
\[
\mathbf{n} \mapsto \mathbf{x}^\mathbf{n} := \prod_{i=0}^\infty x_i^{n_i}.
\]
The subspace of power series without a constant term is a normed vector space with respect to the norm
\[
    ||f(\mathbf{x})|| = \frac{1}{{\rm ord}(f)},
\]
where ${\rm ord}(f)$ is the smallest integer $m$ for which there exists an $\mathbf{n}\in \N^\N_0$ with $|\mathbf{n}|=m$ and the coefficient of $\mathbf{x}^{\mathbf{n}}$ in $f$ is nonzero. This is defined so that the monomials $\{\mathbf{x}^{\mathbf{n}}\}$ form an Schauder basis for $\C[[\mathbf{x}]]$.

The vector space of power series in countably many variables is an extremely flexible object, and by specializing the variables in particular ways we can push properties of this vector space to any of the power series rings we have considered so far.

The following theorem does all the linear algebra and combinatorics necessary for Theorems \ref{thm:explicit_singularities_const_coef} and \ref{thm:explicit_singularities_Frob_coef}. It is most naturally presented in terms of multi-variable power series, which can be specialized as needed.

\begin{theorem}\label{thm:log}
For any $N\ge 1$, the identity
\[
-\log\left(1 - \sum_{i=0}^\infty x_i\right) = -\sum_{\mathbf{n}\in (\N)^{\N}_0-\{\mathbf{0}\}} c_\mathbf{n} \log( 1 - \mathbf{x}^\mathbf{n} )
\]
converges absolutely uniformly in the region cut out by
\[
\sum_{i=0}^\infty |x_i|<1-\epsilon
\]
for any choice of $\epsilon>0$, where the coefficients are given by the following equivalent formulas:
\begin{itemize}
    \item[(i)]
    \[
    c_\mathbf{n} = \frac{1}{|\mathbf{n}|}\sum_{d\mid \gcd(\mathbf{n})} \mu(d) {|\mathbf{n}/d|\choose \mathbf{n}/d},
    \]
    where $|\mathbf{n}|=n_0+n_1+\cdots$, $\gcd(\mathbf{n})=\gcd(n_0,n_1,\cdots)$, and $\mathbf{n}/d=(n_0/d,n_1/d,n_2/d,...)$; and
    \item[(ii)] $c_\mathbf{n}$ equals the number of orbits of length $|\mathbf{n}|$ in
    \[
    \mathcal{P}_{\mathbf{n}}(\Z/|\mathbf{n}|\Z)=\{(A_1,A_2,...) : A_i\subseteq \Z/|\mathbf{n}|\Z,\ |A_i|=n_i,\ A_i\cap A_j = \emptyset\text{ when }i\ne j\}
    \]
    under the action by $\Z/|\mathbf{n}|\Z$ defined by
    \[
    1.(A_1,A_2,...)\mapsto (\{a+1:a\in A_1\},\{a+1:a\in A_2\},...).
    \]
\end{itemize}
In particular, (ii) implies $c_\mathbf{n}\in\N$.
\end{theorem}

This theorem is proven by essentially decomposing the vector on the left hand side in terms of the basis $\{-\log(1-\mathbf{x}^{\mathbf{n}})\}_{\mathbf{n}\in\N^\N_0-\{\mathbf{0}\}}$. This computation is done explicitly via the proof of Lemma \ref{lem:lowertri} in order to produce to description of the coefficients in Theorem \ref{thm:log}(i).

However, the combinatorial description in Theorem \ref{thm:log}(ii) is where the magic really happens. By proving that the coefficients $c_{\mathbf{n}}$ are all natural numbers, we can specialize this theorem to any number of vector spaces of power series (including those in Theorem \ref{thm:explicit_singularities_const_coef} and Theorem \ref{thm:explicit_singularities_Frob_coef}) in order to describe basis decompositions in such vector spaces explicitly in terms of this sequence. This gives us the necessarily leverage to determine when other basis decompositions of power series with respect to a ``log basis" have integer coefficients.

\begin{BOXexample}
    Consider the function $-\log(1-x-y)$. This function has two variables, so it suffices to consider $x_1=x$, $x_2=y$, and $x_i = 0$ for each $i\ge 3$. In particular, $\log(1-\mathbf{x}^{\mathbf{n}}) \ne 0$ if and only if $n_i=0$ for $i\ge 3$.

    We can then write
    \[
    -\log(1-x-y) = -\sum_{\substack{i,j\in\N\\(i,j)\ne (0,0)}} c_{(i,j)}\log(1-x^iy^j).
    \]

    The constants are given explicitly by alternating sums of binomial coefficients
    \[
    c_{(i,j)} = \frac{1}{i+j} \sum_{d\mid \gcd(i,j)} \mu(d){i/d+j/d \choose i/d}.
    \]
    Theorem \ref{thm:log} implies that $c_{(i,j)}$ are all nonnegative integers. A noticeably nice subcase is that whenever $\gcd(i,j)=1$, then
    \[
    c_{(i,j)} = \frac{1}{i+j}{i+j\choose i}.
    \]
    Also nice to know is that
    \[
    c_{(i,0)} = \begin{cases}
        1 & i=1\\
        0 & i > 1.
    \end{cases}
    \]
    We can truncate the summation by putting all terms of sufficiently large degree into an error term. For example:
    \begin{align*}
    \log(1-x-y) =& c_{(1,0)}\log(1-x) + c_{(0,1)} \log(1-y) + c_{(2,0)}\log(1-x^2)\\
    &+ c_{(1,1)}\log(1-xy) + c_{(0,2)}\log(1-y^2) + O(\text{degree }3)\\
    \log(1-x-y)=&\log(1-x) + \log(1-y) + \log(1-xy) + O(\text{degree }3).\\
    \end{align*}
\end{BOXexample}

The right-hand side appears like the log of a product of Euler factors of $L$-functions, meanwhile the left-hand appears like the log of a power series Euler factor. This is what allows us to connect Euler products and zeta products together. The uniform convergence is necessary to produce right halfplanes of convergence. The fact that $c_{\mathbf{n}}\in \N$ is particularly useful, and will allow us to quantify when branch cuts are necessary in the meromorphic continuations.

\begin{proof}
We expand the lefthand side formally via the Taylor series for $-\log(1-T)$ in order to express it in terms of the standard basis $\{\mathbf{x}^\mathbf{n}\}$ of monomials, under which we find that
\begin{align*}
    -\log\left(1 - \sum_{i=0}^\infty x_i\right) &= \sum_{k=1}^{\infty}\frac{1}{k}\left(\sum_{i=0}^\infty x_i\right)^{k}\\
    &=\sum_{k=1}^{\infty}\frac{1}{k}\sum_{\substack{\mathbf{n}\in \N^\N\\|\mathbf{n}|=k}}{k\choose \mathbf{n}}\mathbf{x}^\mathbf{n}\\
    &=\sum_{\mathbf{n}\in (\N)^{\N}_0-\{\mathbf{0}\}} \frac{1}{|\mathbf{n}|}{|\mathbf{n}|\choose \mathbf{n}} \mathbf{x}^\mathbf{n}.
\end{align*}
We formally change bases to $\{-\log(1 - x^\mathbf{n})\}$ by following along the proof of Lemma \ref{lem:lowertri}. The Taylor series decomposition of the basis vectors is given by
\[
-\log(1 - \mathbf{x}^\mathbf{n}) = \sum_{k=1}^{\infty} \frac{1}{k} \mathbf{x}^{k\mathbf{n}},
\]
where $k\mathbf{n}=(kn_0,kn_1,kn_2,...)$. Thus, the linear transformation sending $\mathbf{x}^\mathbf{n}$ to $-\log(1-\mathbf{x}^\mathbf{n})$ has a matrix given by $A=[a_{\mathbf{n},\mathbf{m}}]$ for
\[
a_{\mathbf{n},\mathbf{m}} = \begin{cases}
\frac{1}{d} & \mathbf{n} = d\mathbf{m}\\
0 & \text{else}.
\end{cases}
\]
The matrix $A$ being lower triangular in the sense of Lemma \ref{lem:lowertri}, but we will need to compute the inverse matrix explicitely. We do so by solving the equation $AB = I$. Write $B=[b_{\mathbf{n},\mathbf{m}}]$ and $I = [\delta_{\mathbf{n},\mathbf{m}}]$ the identity matrix. This gives a system of equations
\[
\delta_{\mathbf{n},\mathbf{m}} = \sum_{\mathbf{j}\in \N^{\N}_0-\{\mathbf{0}\}} a_{\mathbf{n},\mathbf{j}}b_{\mathbf{j},\mathbf{m}} = \sum_{d\mid \mathbf{n}} \frac{1}{d} b_{\frac{\mathbf{n}}{d},\mathbf{m}},
\]
where we write $d\mid \mathbf{n}$ to mean $d\mid n_i$ for each $i$ (or in other words, there exists a $\mathbf{j}\in \N^{\N}_0$ with $\mathbf{n} = d\mathbf{j}$). This is a Dirichlet convolution between the function $d\mapsto 1/d$ and $d\mapsto b_{\frac{\mathbf{n}}{\gcd(\mathbf{n})}*d,\mathbf{m}}$ evaluated at $\gcd(\mathbf{n})$, where we define $\gcd(\mathbf{n}) = \gcd(n_1,n_2,...)$ and note that $d\mid \mathbf{n}$ iff $d\mid \gcd(\mathbf{n})$. The Dirichlet inverse of $d\mapsto 1/d$ is $d\mapsto \mu(d)/d$ because the function is completely multiplicative, so by convoluting on both sides by this function we find that
\begin{align*}
b_{\mathbf{n},\mathbf{m}} = b_{\frac{\mathbf{n}}{\gcd(\mathbf{n})}*\gcd(\mathbf{n}),\mathbf{m}} &= \sum_{d\mid\gcd(\mathbf{n})}\frac{\mu(d)}{d}\delta_{\frac{\mathbf{n}}{\gcd(\mathbf{n})}*\frac{\gcd(\mathbf{n})}{d},\mathbf{m}}\\
&=\begin{cases}
\frac{\mu(d)}{d} & \mathbf{n}= d\mathbf{m}\\
0 & \text{else}.
\end{cases}
\end{align*}
We have successfully inverted the linear transformation, where we remark that the formula is particularly nice falling out of a M\"obius inversion.

We now apply the $\{x^\mathbf{n}\}\to\{-\log(1-\mathbf{x}^\mathbf{n})\}$ change of basis matrix to the $\{\mathbf{x}^\mathbf{n}\}$-coefficients of $-\log(1-\sum x_i)$ to get
\begin{align*}
c_{\mathbf{n}} &= \sum_{\mathbf{m}\in \N^N-\{\mathbf{0}\}} b_{\mathbf{n},\mathbf{m}}\frac{1}{|\mathbf{m}|}{|\mathbf{m}|\choose \mathbf{m}}\\
&= \sum_{d\mid \gcd(\mathbf{n})} \frac{\mu(d)}{d} \frac{1}{|\mathbf{n}/d|}{|\mathbf{n}/d|\choose \mathbf{n}/d}\\
&=\frac{1}{|\mathbf{n}|}\sum_{d\mid \gcd(\mathbf{n})} \mu(d){|\mathbf{n}/d|\choose \mathbf{n}/d}.
\end{align*}
This formally proves (i).

Part (ii) is simple to check given the combinatorial construction. Consider that $\mathcal{P}_\mathbf{n}(\Z/|\mathbf{n}|\Z)$ has ${|\mathbf{n}|\choose \mathbf{n}}$ elements by definition of the multinomial coefficient. Letting $f(d)$ denote the number of orbits of length $d$ (so necessarily $f(d)\ne 0$ implies $d\mid |\mathbf{n}|$), we note that $A\in \mathcal{P}_\mathbf{n}(\Z/|\mathbf{n}|\Z)$ has an orbit of length $d$ iff $|\Stab_{\Z/|\mathbf{n}|\Z}(A)|=|\mathbf{n}|/d$, or equivalently $\Stab_{\Z/|\mathbf{n}|\Z}(A) = d\Z/|\mathbf{n}|\Z$ as the orders of subgroups of $\Z/|\mathbf{n}|\Z$ are unique. Thus
\begin{align*}
|\mathcal{P}_\mathbf{n}(\Z/|\mathbf{n}|\Z)| &=|\operatorname{Fix}_{\mathcal{P}_\mathbf{n}(\Z/|\mathbf{n}|\Z)}(|\mathbf{n}|)|\\
&=\sum_{d\mid |\mathbf{n}|}\#\{A\in \mathcal{P}_\mathbf{n}(\Z/|\mathbf{n}|\Z) : \Stab_{\Z/|\mathbf{n}|\Z}(A)=d\Z/|\mathbf{n}|\Z\}\\
&=\sum_{d\mid |n|} d\cdot f(d)
\end{align*}
By M\"obius inversion, we find that
\[
|\mathbf{n}|\cdot f(|\mathbf{n}|) = \sum_{d\mid |\mathbf{n}|} \mu(d) |\operatorname{Fix}_{\mathcal{P}_{\mathbf{n}}(\Z/|\mathbf{n}|\Z)}(|\mathbf{n}|/d)|.
\]
We now completely describe the sets $\operatorname{Fix}_{\mathcal{P}_{\mathbf{n}}(\Z/|\mathbf{n}|\Z)}(|\mathbf{n}|/d)$. An element $A=(A_1,...,A_N)\in \mathcal{P}_{\mathbf{n}}(\Z/|\mathbf{n}|\Z)$ is fixed by $|\mathbf{n}|/d$ if and only if whenever $a\in A_i$, then so too is the arithmetic progression $a, a+\frac{|\mathbf{n}|}{d}, a+ 2\frac{|\mathbf{n}|}{d}, ...$. This implies two things for us:
\begin{itemize}
    \item First, it implies that each $A_i$ is a disjoint union of arithmetic progressions of length $\frac{|\mathbf{n}|}{|\mathbf{n}|/d}=d$, so in particular $d\mid |A_i|=n_i$. Thus, we have proven that $|\operatorname{Fix}_{\mathcal{P}_{\mathbf{n}}(\Z/|\mathbf{n}|\Z)}(|\mathbf{n}|/d)|=0$ if $d\nmid \gcd(\mathbf{n})$,
    \item Second, it implies that $A$ must be in the image of the embedding
    \[
    \pi^*:\mathcal{P}_{\mathbf{n}/d}\left(\Z/\frac{|\mathbf{n}|}{d}\Z\right) \to \mathcal{P}_\mathbf{n}(\Z/|\mathbf{n}|\Z)
    \]
    given by $(B_1,B_2...)\mapsto (\pi^{-1}(B_1),\pi^{-1}(B_2),...)$ for $\pi:\Z/|\mathbf{n}|\Z\to\Z/\frac{|\mathbf{n}|}{d}\Z$ the quotient map.
\end{itemize}
It is trivial to see that the converse holds as well, proving that if $d\mid \gcd(\mathbf{n})$ then
\begin{align*}
|\operatorname{Fix}_{\mathcal{P}_{\mathbf{n}}(\Z/|\mathbf{n}|\Z)}(|\mathbf{n}|/d)| &= |\im(\pi^*)|\\
&= \left\lvert\mathcal{P}_{\mathbf{n}/d}\left(\Z/\frac{|\mathbf{n}|}{d}\Z\right)\right\rvert\\
&= {|\mathbf{n}/d|\choose \mathbf{n}/d},
\end{align*}
and is $0$ otherwise. Thus we have proven that
\[
|\mathbf{n}|\cdot f(|\mathbf{n}|) = \sum_{d\mid \gcd(\mathbf{n})} \mu(d){|\mathbf{n}/d|\choose \mathbf{n}/d} = |\mathbf{n}|c_\mathbf{n},
\]
concluding the proof of (ii).

Lastly, we must prove uniform convergence. We can do this using some crude upper bounds. We certainly have by (ii)
\[
c_\mathbf{n} \le |\mathcal{P}_\mathbf{n}(\Z/|\mathbf{n}|\Z)| = {|\mathbf{n}|\choose \mathbf{n}}.
\]
So we consider the upper bound
\begin{align*}
\sum_{\mathbf{n}\in \N^{\N}_0-\{\mathbf{0}\}} {|\mathbf{n}|\choose \mathbf{n}} |-\log (1 - \mathbf{x}^\mathbf{n}))|.
\end{align*}
When $|x_j|\le \sum |x_i|<1-\epsilon$ this implies $|\mathbf{x}^\mathbf{n}|<(1-\epsilon)^{|\mathbf{n}|}\le 1-\epsilon < 1$, so that the remainder theorem for Taylor series implies that $-\log(1-\mathbf{x}^\mathbf{n}) = \mathbf{x}^\mathbf{n} + O(M\mathbf{x}^{2\mathbf{n}})$ where
\[
M = \sup_{|t|<1-\epsilon} (1-t)^{-2} \le \epsilon^2.
\]
Thus, $|\log(1-\mathbf{x}^\mathbf{n})| = O(|\mathbf{x}^\mathbf{n}|)$. We use this to produce an upper bound
\begin{align*}
\ll \sum_{\mathbf{n}\in \N^{\N}_0-\{\mathbf{0}\}} {|\mathbf{n}|\choose \mathbf{n}} |\mathbf{x}^\mathbf{n}| &=\sum_{k=1}^{\infty} \left(\sum_{i=0}^\infty|x_i|\right)^k.
\end{align*}
This is a geometric series, which converges uniformly on the region $\sum |x_i|<1-\epsilon$ as required.
\end{proof}

\section{Cyclic Tensor Powers}\label{sec:cyclictensor}

It is necessary for Theorem \ref{thm:explicit_singularities_Frob_coef} to consider the following construction:

\begin{definition}\label{def:cyclic_tensor}
Let $\rho:G\to GL(V)$ be a finite dimensional representation of $G$ over the $\C$-vector space $V$. For any positive integer $n$, fix a primitive $n^{\rm th}$ root of unity $\zeta$. Define the \textbf{cyclic tensor power} of $\rho$ by $n$ to be $\rho^{\otimes_{\text{cyc}} n}:G\to GL(V^{\otimes_{\text{cyc}} n})$, where
\[
V^{\otimes_{\text{cyc}} n}:= V^{\otimes n}/\langle (\zeta^m - m).v\rangle,
\]
where $\Z/n\Z$ acts on $V^{\otimes n}$ by $m.(a_1\otimes \cdots \otimes a_{n})=a_{1+m}\otimes \cdots \otimes a_{n+m}$ under the identification $a_i=a_j$ whenever $i\equiv j\mod n$.
\end{definition}

One readily verifies that $V^{\otimes_{\text{cyc}} n}$ is a quotient $G$-module of the tensor power $V^{\otimes n}$ as the $G$ action commutes with scalar multiples and permuting the entries of pure tensors.

While it is unnecessary for our purposes, we remark that the isomorphism class of $V^{\otimes_{\text{cyc}} n}$ is independent of the choice of primitive root $\zeta$. Indeed, any other choice of primitive root of unity is given by $\zeta^m$ for some $m$ coprime to $n$ and the isomorphism $\{m\}:V^{\otimes n}\to V^{\otimes n}$ defined by $v_1\otimes \cdots v_n \mapsto  v_{1\cdot m} \otimes \cdots \otimes v_{n\cdot m}$ induced by the permutation action of $\Z/n\Z$ sends
\begin{align*}
(\zeta^d - [d]).v &\to (\zeta^d \{m\} - [d\cdot m]\{m\}).v\\
&= ((\zeta^{m^{-1}})^{dm} - [dm]).\{m\}v.
\end{align*}
This implies $\{m\}$ sends $V^{\otimes_{\text{cyc}} n}$ to the corresponding quotient with $\zeta^{m^{-1}}$ as the chosen primitive root. By applying the same argument to $\{m^{-1}\}$ (which exists modulo $n$), we conclude that this is an isomorphism and that the choice of primitive root does not matter.

\begin{BOXexample}
    Let $\rho:G\to {\rm GL}(V)$ be a representation. Let's examine the cyclic tensor powers by small integers.
    \begin{enumerate}[(1)]
        \item $\rho^{\otimes_{\rm cyc} 1} = \rho$. Indeed, the only primitive first root of unity is $1$ itself, while $\Z/1\Z$ acts trivially on anything. Thus
        \[
            V^{\otimes 1}/\langle (1^1 - 1).v\rangle = V/\langle 0\rangle \cong V.
        \]
        \item $\rho^{\otimes_{\rm cyc} 2}=\Lambda_2\rho$ is the representation associated to the exterior algebra $\Lambda_2(V)$. This is because
        \[
            V^{\otimes 2}/\langle ((-1)^1 - 1).(v\otimes w)\rangle = V^{\otimes 2}/\langle -v\otimes w - w\otimes v\rangle = \Lambda_2 V.
        \]
        \item $\rho^{\otimes_{\rm cyc} 3}$ is \emph{not} isomorphic to $\Lambda_3\rho$. The construction becomes more subtle from $3^{\rm rd}$ powers onwards. Inside of $V^{\otimes_{\rm cyc} 3}$, the classes of $u\otimes v\otimes w$ and $v\otimes u\otimes w$ need not be related, unlike in the exterior algebra. However, there is a cyclic relationship between the following classes:
        \[
            [u\otimes v\otimes w] = \zeta_3 [w\otimes u\otimes v].
        \]
        The appearance of a complex coefficient in this relationship is also distinct from the exterior algebra.
    \end{enumerate}
\end{BOXexample}

The cyclic tensor power can be used to identify the character ${\rm tr}(\rho(g))$ using only characteristic polynomials.

\begin{theorem}\label{thm:log_trchar}
Let $\rho:G\to GL(V)$ be a finite dimensional representation of $G$ over the $\C$-vector space $V$. Then
\[
-\log( 1 - {\rm tr}(\rho(g))z ) = -\sum_{n=1}^{\infty}\log(\det(I-\rho^{\otimes_{\text{cyc}} n}(g)z^n)),
\]
where the sum converges absolutely in the region $|z|<1$ and uniformly on the region $|z|<1-\epsilon$.
\end{theorem}

While $1-{\rm tr}(\rho(g))z$ is not itself a characteristic polynomial, this result shows that it equals the convergent product of infinitely many characteristic polynomials. This will be immensely useful for getting a handle on the sequence $b_{n,\rho}(Q)$ appearing in Theorem \ref{thm:explicit_singularities_Frob_coef}.

\begin{proof}[Proof of Theorem \ref{thm:log_trchar}]
We utilize a standard fact for characteristic polynomials of representations, that
\[
-\log(\det(I-\varphi(g)z)) = \sum_{k=1}^{\infty} \frac{{\rm tr}(\varphi(g^k))}{k} z^k
\]
for $|z|<1$.

Thus, it follows that
\begin{align*}
-\sum_{n=1}^{\infty}\log(\det(I-\rho^{\otimes_{\text{cyc}} n}(g)z^n)) &= \sum_{n=1}^{\infty}\sum_{k=1}^{\infty} \frac{{\rm tr}(\rho^{\otimes_{\text{cyc}} n}(g^k))}{k}z^{nk}\\
&=\sum_{k=1}^{\infty}\left(\sum_{d\mid k} \frac{{\rm tr}(\rho^{\otimes_{\text{cyc}} d}(g^{k/d}))}{k/d}\right)z^k.
\end{align*}
In order to prove the equality formally, it suffices to prove that the $k^{\rm th}$ coefficient is equal to ${\rm tr}(\rho(g))^{k}/k$.

Fix a basis for $V$, and therefore a basis of pure tensors of any tensor power labelled $\mathcal{B}_{\otimes d}$. Choose a representative of each nonzero equivalence class of pure tensor powers to form a basis for $V^{\otimes_{\text{cyc}} d}$ which we label $\mathcal{B}_{\otimes_{\text{cyc}} d}$. To compute the trace of $\rho^{\otimes_{\text{cyc}} d}(g)$, it suffices to know the coefficient of the basis element $[a]$ in the vector $\rho^{\otimes_{\text{cyc}} d}(g)[a]$, which we denote $\langle \rho^{\otimes_{\text{cyc}} d}(g)[a],[a]\rangle$, implying
\begin{align*}
{\rm tr}(\rho^{\otimes_{\text{cyc}}d}(g)) &= \sum_{\substack{[a]\in\mathcal{B}_{\otimes_{\text{cyc}} d}}} \langle\rho^{\otimes_{\text{cyc}} d}(g)[a],[a]\rangle.
\end{align*}
We note that $\rho^{\otimes_{\text{cyc}} d}(g)[a] = [\rho^{\otimes d}(g) a]$ and that $[m.a]=\zeta^m[a]$ so that the representatives of $[a]$ are all of the form $\zeta^{-m}(m.a)$. Thus
\begin{align*}
\langle \rho^{\otimes_{\text{cyc}} d}(g)[a],[a]\rangle &= \langle [\rho^{\otimes d}(g)a],[a]\rangle\\
&= \frac{1}{d}\sum_{m=0}^{d-1} \langle \rho^{\otimes d}(g)a,\zeta^{-m}(m.a)\rangle\\
&= \frac{1}{d}\sum_{m=0}^{d-1} \zeta^{-m}\langle \rho^{\otimes d}(g)a,m.a\rangle
\end{align*}
We seek to expand the summation to a sum over $\mathcal{B}_{\otimes d}$, essentially undoing the quotient defining $V^{\otimes_{\text{cyc}}d}$. A pure tensor $a=a_1\otimes \cdots \otimes a_n$ has $[a]\ne 0$ if and only if $\Stab_{\Z/d\Z}(a)=1$, which implies the orbit under the $\Z/d\Z$ action is necessarily of size $d$. We remark on two things:
\begin{enumerate}
    \item{If $a$ is a representative of the basis element $[a]$ of $V^{\otimes_{\text{cyc}} d}$ with $\Stab_{\Z/d\Z}(a) = 1$, then
    \[
    \frac{\langle \rho^{\otimes_{\text{cyc}} d}(g)[m.a],[m.a]\rangle}{\langle [m.a],[m.a]\rangle} = \frac{\langle \zeta^m\rho^{\otimes_{\text{cyc}} d}(g)[a],\zeta^m[a]\rangle}{\langle \zeta^m[a],\zeta^m[a]\rangle} = \langle \rho^{\otimes_{\text{cyc}} d}(g)[a],[a]\rangle
    \]
    is independent of the choice of representative under the $\Z/d\Z$ action. (Note that $[m.a]$ is not a basis element, so it is important that we include it in the denominator).}
    \item{If $a$ is a pure tensor with $\Stab_{\Z/d\Z}(a) = \ell\Z/d\Z$ for $\ell\ne d$, then
    \begin{align*}
    \sum_{m=0}^{d-1} \zeta^{-m}\langle \rho^{\otimes d}(g)a,m.a\rangle &= \sum_{m=0}^{\ell-1}\langle \rho^{\otimes d}(g)a,m.a\rangle \left(\sum_{j=0}^{d/\ell - 1} \zeta^{-(m+j\ell)}\right)\\
    &=\sum_{m=0}^{\ell-1}\zeta^{-m}\langle \rho^{\otimes d}(g)a,m.a\rangle \left(\sum_{j=0}^{d/\ell - 1} (\zeta^{-\ell})^j\right)\\
    &= 0,
    \end{align*}
    as we note that $\zeta^{-\ell}$ is a primitive $(d/\ell)^{\rm th}$ root of unity and $d/\ell\ne 1$.}
\end{enumerate}

Noting that the trace is the same for all bases, we partition $\mathcal{B}_{\otimes d}$ into $d$ lifts of bases for $V^{\otimes_{\text{cyc}} d}$ (which can be gotten from the orbits of a lift of $\mathcal{B}_{\otimes_{\text{cyc}}d}$) together with the collection of pure tensors $a$ with $\Stab_{\Z/d\Z}(a)\ne 1$. This produces the sum
\begin{align*}
{\rm tr}(\rho^{\otimes_{\rm cyc} d}(g))&=\frac{1}{d}\sum_{a\in\mathcal{B}_{\otimes d}}\sum_{m=0}^{d-1} \zeta^{-m} \langle\rho^{\otimes d}(g)a,m.a\rangle\\
&=\frac{1}{d}\sum_{m=0}^{d-1}\zeta^{-m} \sum_{a\in\mathcal{B}_{\otimes d}} \langle\rho^{\otimes d}(g)a,m.a\rangle.
\end{align*}
We next claim that the inner summation depends only on $\gcd(d,m)$. Indeed, for any $n$ coprime to $d$, consider the isomorphism $\{n\}:V^{\otimes d} \to V^{\otimes d}$ sending
\[
a_1\otimes \cdots \otimes a_d\mapsto a_n\otimes \cdots \otimes a_{nd}.
\]
This map certainly preserves pure tensors and commutes with $\rho^{\otimes d}(g)$ (as $\rho^{\otimes d}(g)$ acts coordinate wise). We also note that
\[
\langle \rho^{\otimes d}(g)\{n^{-1}\}a,m.\{n^{-1}\}a\rangle = \langle\{n^{-1}\} \rho^{\otimes d}(g)a,\{n^{-1}\}((nm).a)\rangle = \langle\rho^{\otimes d}(g)a,((nm).a)\rangle,
\]
where the last equality follows from $\{n^{-1}\}$ being an isomorphism which restricts to a bijection on the basis of pure tensors. Thus, we have proven
\[
\sum_{\substack{a\in\mathcal{B}_{\otimes d}\\\Stab_{\Z/d\Z}(a)=1}} \langle\rho^{\otimes d}(g)a,m.a\rangle = \sum_{\substack{a\in\mathcal{B}_{\otimes d}\\\Stab_{\Z/d\Z}(a)=1}} \langle\rho^{\otimes d}(g)a,((nm).a)\rangle
\]
whenever $\gcd(n,d)=1$, showing that this sum depends only on $\gcd(d,m)$. Partitioning as such, we find that
\begin{align*}
{\rm tr}(\rho^{\otimes_{\rm cyc} d}(g)) = \frac{1}{d}\sum_{q\mid d}\sum_{a\in\mathcal{B}_{\otimes d}} \langle\rho^{\otimes d}(g)a,q.a\rangle\left(\sum_{\substack{0\le m < d\\\gcd(m,d)=q}}\zeta^{-m} \right).
\end{align*}
The sum of roots of unity can be evaluated using the M\"obius function trick
\begin{align*}
\sum_{\substack{0\le m < d\\\gcd(m,d)=q}}\zeta^{-m} &= \sum_{\substack{0\le m < d\\q\mid m\\\gcd(m/q,d/q)=1}} \zeta^{-m}\\
&=\sum_{\substack{0\le m< d\\ q\mid m}} \zeta^{-m} \left(\sum_{\substack{\ell\mid d/q\\\ell\mid m/q}} \mu(\ell)\right)\\
&= \sum_{\ell\mid d/q} \mu(\ell) \sum_{r=0}^{d/q\ell - 1} \zeta^{-rq\ell}\\
&= \mu(d/q).
\end{align*}
Therefore we conclude that
\[
{\rm tr}(\rho^{\otimes_{\rm cyc} d}(g)) = \frac{1}{d}\sum_{q\mid d}\mu(d/q)\sum_{a\in\mathcal{B}_{\otimes d}} \langle\rho^{\otimes d}(g)a,q.a\rangle.
\]
We choose an identification $(\mathcal{B}_{\otimes q})^{d/q} \leftrightarrow \mathcal{B}_{\otimes d}$ via
\[
b_i = a_{1+iq} \otimes a_{2+iq} \otimes \cdots \otimes a_{q+iq}
\]
for each $i=1,...,d/q$. The action by $q$ on $a$ sends $b_i$ to $b_{i+1}$, taking $i$ modulo $d/q$, so we can write
\[
\langle\rho^{\otimes d}(g)a,q.a\rangle = \prod_{i=1}^{d/q} \langle \rho^{\otimes q}(g) b_i, b_{i+1}\rangle.
\]
We make use of a general fact about powers of matrices: If $A$ is an invertible matrix on a finite dimensional vector space $V$ with basis $\mathcal{B}$, then
\[
\sum_{a\in \mathcal{B}}\langle A^n a, v\rangle = \sum_{a_1,...,a_n\in \mathcal{B}} \langle A a_1,a_2\rangle\langle A a_2,a_3\rangle\cdots \langle A a_n,v\rangle.
\]
Thus,
\begin{align*}
\sum_{a\in\mathcal{B}_{\otimes d}} \langle\rho^{\otimes d}(g)a,q.a\rangle &= \sum_{b_1,...,b_{q}\in\mathcal{B}_{\otimes d/q}} \prod_{i=1}^{d/q}\langle \rho^{\otimes q}(g)b_i,b_{i+1}\rangle\\
&= \sum_{b\in\mathcal{B}_{\otimes q}}\langle \rho^{\otimes q}(g)^{d/q} b,b\rangle\\
&= {\rm tr}(\rho(g^{d/q}))^{q}
\end{align*}
and so
\[
{\rm tr}(\rho^{\otimes_{\text{cyc}} d}(g)) = \frac{1}{d}\sum_{q\mid d}\mu(d/q){\rm tr}(\rho(g^{d/q}))^q.
\]
We conclude the proof by evaluating
\begin{align*}
-\sum_{n=1}^{\infty}\log(\det(I-\rho^{\otimes_{\text{cyc}} n}(g)z^n)) &=\sum_{k=1}^{\infty}\left(\sum_{d\mid k} \frac{{\rm tr}(\rho^{\otimes_{\text{cyc}} d}(g^{k/d}))}{k/d}\right)z^k\\
&=\sum_{k=1}^{\infty}\frac{z^k}{k}\left(\sum_{d\mid k} \sum_{q\mid d}\mu(d/q) {\rm tr}(\rho((g^{k/d})^{d/q}))^q\right)\\
&=\sum_{k=1}^{\infty}\frac{z^k}{k}\left(\sum_{tyz = k}\mu(y) {\rm tr}(\rho((g^{k/ty})^{ty/t}))^t\right)\\
&=\sum_{k=1}^{\infty}\frac{z^k}{k}\left(\sum_{tyz = k}\mu(y) {\rm tr}(\rho((g^{k/t})))^t\right)\\
&=\sum_{k=1}^{\infty}\frac{z^k}{k} \sum_{t\mid k}{\rm tr}(\rho((g^{k/t})))^t\left(\sum_{y\mid k/t} \mu(y)\right)\\
&=\sum_{k=1}^{\infty}\frac{z^k}{k} {\rm tr}(\rho(g^{k/k}))^k\\
&= -\log(1 - {\rm tr}\rho(g)z).
\end{align*}
\end{proof}

\section{Proofs of Explicit Results}

We are now prepared to proof Theorem \ref{thm:explicit_singularities_const_coef} and Theorem \ref{thm:explicit_singularities_Frob_coef}.

\begin{proof}[Proof of Theorem \ref{thm:explicit_singularities_const_coef}]
    Part (i) follows from the fact that the linear transformation $z\C[[z]]\to z\C[[z]]$ sending $z^n\mapsto -\log(1-z^n)$ is an invertible lower triangular matrix, as was used in the proof of Proposition \ref{prop:basis_const}. The recursive formula then follows from the inductive construction in the proof of Lemma \ref{lem:lowertri}.

    Part (ii) will follow from specializing Theorem \ref{thm:log}. Firstly, we specialize to $x_i = q_iz^i$ so that
    \begin{align*}
        \log Q(z) &= -\sum_{\mathbf{n}\in \N^\N_0-\{\mathbf{0}\}} c_\mathbf{n} \log\left(1 - \prod (q_i z^i)^{n_i}\right)\\
        &=-\sum_{k=1}^\infty\sum_{\substack{\mathbf{n}\in \N^\N_0-\{\mathbf{0}\}\\\sum in_i = k}} c_\mathbf{n} \log\left(1 - \mathbf{q}^{\mathbf{n}} z^k\right).
    \end{align*}
    Next, we take advantage of the fact that $\mathbf{q}^\mathbf{n}$ is an integer. If $\mathbf{q}^\mathbf{n}=0$ this term vanishes. If $\mathbf{q}^\mathbf{n}>0$, then we can specialize Theorem \ref{thm:log} to $x_i = 1$ for $i=1,2,...,\mathbf{q}^\mathbf{n}$ and $x_i = 0$ for $i>\mathbf{q}^\mathbf{n}$. This implies
    \begin{align*}
        \log\left(1 - \mathbf{q}^{\mathbf{n}} z^k\right) &= -\sum_{\substack{\mathbf{m}\in \N^\N_0-\{\mathbf{0}\}\\i>\mathbf{q}^{\mathbf{n}}\Rightarrow m_i =0}} c_\mathbf{m}\log(1 - z^{k|\mathbf{m}|}).
    \end{align*}
    If $\mathbf{q}^\mathbf{n}<0$, then we likewise specialize Theorem \ref{thm:log} to $x_i = -1$ for $i=1,2,...,|\mathbf{q}^\mathbf{n}|$ and $x_i = 0$ for $i>\mathbf{q}^\mathbf{n}$. This implies
    \begin{align*}
        \log\left(1 - \mathbf{q}^{\mathbf{n}} z^k\right) &= -\sum_{\substack{\mathbf{m}\in \N^\N_0-\{\mathbf{0}\}\\i>|\mathbf{q}^{\mathbf{n}}|\Rightarrow m_i =0}} c_\mathbf{m}\log(1 - (-1)^{|\mathbf{m}|}z^{k|\mathbf{m}|})\\
        &=-\sum_{\substack{\mathbf{m}\in \N^\N_0-\{\mathbf{0}\}\\i>|\mathbf{q}^{\mathbf{n}}|\Rightarrow m_i =0}} c_\mathbf{m}\log(1 - z^{k|\mathbf{m}|})+
        \sum_{\substack{\mathbf{m}\in \N^\N_0-\{\mathbf{0}\}\\i>|\mathbf{q}^{\mathbf{n}}|\Rightarrow m_i =0\\2\nmid |\mathbf{m}|}} c_\mathbf{m}\log(1 - z^{2k|\mathbf{m}|}).
    \end{align*}
    By composing these linear combinations and combining like terms, we find that the coefficients $b_n$ of $-\log(1-z^n)$ in this decomposition are comprised of sums and differences of $c_{\mathbf{n}}c_{\mathbf{m}}$ for various $\mathbf{n},\mathbf{m}\in \N^\N_0$. Theorem \ref{thm:log}(ii) implies that these are integers, so sums and differences of these values are necessarily also integers. Thus $b_n\in \Z$, concluding the proof.
\end{proof}

The proof of Theorem \ref{thm:explicit_singularities_Frob_coef} is similar, although we need to use a slightly different linear transformation.

\begin{proof}[Proof of Theorem \ref{thm:explicit_singularities_Frob_coef}]
    Unlike in Proposition \ref{prop:basis_Frob}, we do not directly consider the linear transformation sending ${\rm tr}\rho(\Fr_p) z^n\mapsto -\log\det(I-\rho(\Fr_p)z^n)$. Instead, we break this transformation into two pieces.

    First, consider the linear transformation $zR(\Gal(E/\Q))[[z]]\to zR(\Gal(E/\Q))[[x]]$ sending ${\rm tr}\rho(\Fr_p) z^n\mapsto -\log(1-{\rm tr}\rho(\Fr_p)z^n)$. Let $T=[t_{(n,\rho),(m,\gamma)}]$ be the matrix of this transformation with respect to the standard basis. By expanding
    \begin{align*}
        -\log(1-{\rm tr}\rho(\Fr_p)z^n) &= \sum_{d=1}^{\infty} \frac{1}{n}{\rm tr}\rho(\Fr_p)^dz^{dn}\\
        &=\sum_{d=1}^{\infty} \frac{1}{n}{\rm tr}\rho^{\otimes d}(\Fr_p)z^{dn}\\
        &=\sum_{d=1}^{\infty}\sum_{\gamma\text{ irrep.}}\frac{\langle \rho^{\otimes d},\gamma\rangle}{n}{\rm tr}\gamma(\Fr_p)z^{dn},
    \end{align*}
    we can determine that the coefficients are given by
    \[
        t_{(n,\rho),(m,\gamma)} = \begin{cases}
            \displaystyle \sum_{\gamma\text{ irrep.}}\frac{\langle \rho^{\otimes m/n},\gamma\rangle}{n} & n\mid m\\
            0 & \text{else.}
        \end{cases}
    \]
    Thus, $T$ is a lower triangular matrix with nonzero diagonal. Following the construction solving a linear system in the proof of Lemma \ref{lem:lowertri}, the sequence $r_{n,\rho}$ determined by the basis decomposition
    \[
        \log Q_p(z) = -\sum_{n,\rho}r_{n,\rho}\log(1-{\rm tr}\rho(\Fr_p)z^n)
    \]
    is given precisely by the recurrence relation in part (i).
    
    By Theorem \ref{thm:log_trchar}, we know that
    \begin{align*}
        \log(1-{\rm tr}\rho(\Fr_p)z^n) &= \sum_{d=1}^{\infty} \log\det(I-\rho^{\otimes_{\rm cyc} d}(\Fr_p)z^{dn})\\
        &=\sum_{d=1}^{\infty} \sum_{\gamma\text{ irrep.}}\langle\rho^{\otimes_{\rm cyc} d},\gamma\rangle\log\det(I-\gamma(\Fr_p)z^{dn}).
    \end{align*}
    Plugging this in above and combining like terms concludes the proof of part (i).

    Part (ii) is proven very similarly to Theorem \ref{thm:explicit_singularities_const_coef}(ii). When $\log Q_p(z) \in zR(\Gal(E/\Q),\Z)[[z]]$, the sequence $r_{n,\rho}$ are determined to be integers by exactly the same argument as for Theorem \ref{thm:explicit_singularities_const_coef}(ii), together with the fact that $\langle \rho^{\otimes d},\gamma\rangle$ is an integer (given by the number of times $\gamma$ appears as an irreducible component of $\rho^{\otimes d}$). Similarly, $\langle \gamma^{\otimes_{\rm cyc} d},\rho\rangle$ is always an integer so that part (i) implies $b_{n,\rho}\in \Z$.
\end{proof}

\bibliographystyle{alpha}

\bibliography{main_MeromEulerProd.bbl}

\end{document}